\documentclass[12pt,a4paper]{extarticle}
\usepackage[english]{babel}
\usepackage[letterpaper,top=2cm,bottom=2cm,left=2.5cm,right=2.5cm,marginparwidth=1.75cm]{geometry}
\usepackage{titlesec}
\usepackage{abstract}
\usepackage{booktabs}  
\usepackage{tikz-cd}

\titleformat{\section}
  {\scshape\centering}
  {\thesection.}
  {1em}
  {}

\titleformat{\subsection}
  {\normalfont\bfseries}
  {\thesubsection.}
  {1em}
  {}

\titleformat{\subsubsection}
  {\normalfont\itshape}
  {\thesubsubsection.}
  {1em}
  {}
\usepackage{amsmath, amsfonts, amsthm}
\usepackage{graphicx}
\usepackage{hyperref}
\hypersetup{
  colorlinks=true,
  linkcolor=blue,
  citecolor=teal,
  urlcolor=magenta
}
\usepackage{mathtools}
\usepackage{amssymb} 
\usepackage[backend=bibtex,style=alphabetic,maxbibnames=99]{biblatex}
\allowdisplaybreaks 
\usepackage{comment} 
\usepackage[capitalize]{cleveref}
\usepackage{bbm}
\usepackage{nicematrix}
\usepackage{tikz}
\usepackage[ruled,linesnumbered]{algorithm2e}
\usepackage{microtype}

\newtheorem{theorem}{Theorem}[section]
\newtheorem{lemma}[theorem]{Lemma}
\newtheorem{corollary}[theorem]{Corollary}
\newtheorem{proposition}[theorem]{Proposition}

\newtheorem{problem}[theorem]{Problem}
\theoremstyle{definition}
\newtheorem{definition}[theorem]{Definition}
\newtheorem{remark}[theorem]{Remark}

\newtheorem{example}[theorem]{Example}

\newcommand{\Spec}{\mathrm{Spec}}

\newcommand{\KK}{\mathbb{C}}
\newcommand{\Ab}{\mathbb{A}}
\newcommand{\CC}{\mathbb{C}}
\newcommand{\Z}{\mathbb{Z}}
\newcommand{\PP}{\mathbb{P}}

\newcommand{\Gr}{\mathrm{Gr}}
\newcommand{\pr}{\mathrm{pr}}

\newcommand{\NN}{\mathcal{N}}
\newcommand{\XX}{\mathcal{X}}

\newcommand{\codim}{\operatorname{codim}}

\usepackage{xcolor}

\title{ \bf  Degenerating Discriminants }
\author{Viktoriia Borovik and Clara Briand}
\date{}

\begin{document}
\maketitle
\begin{abstract}
We study the behavior of projective dual varieties and discriminants under flat~degenerations. 
For Gröbner degenerations, we show that the conormal variety admits a Gröbner degeneration with opposite weights on the dual coordinates. Using Whitney stratifications and Sabbah’s formula, we describe the irreducible components and multiplicities of the special fiber of this degeneration, and hence of the limiting dual hypersurface. We then extend these results to higher associated hypersurfaces in Grassmannians. As applications, we recover classical formulas for hypersurfaces with isolated singularities, analyze degenerations of generic complete intersections and reciprocal linear spaces, and relate the theory to mixed discriminants of parametrized polynomial~systems.
\end{abstract}


\section{Introduction}
Discriminants describe the parameter values for which a polynomial system
fails to be generic. For a parametric system
$
        f_1(x)=\cdots=f_n(x)=0
$,
the discriminantal locus records when the system has a multiple root.
In the sparse case, where the equations have fixed monomial supports,
this leads to the classical $A$-discriminants of
Gelfand-Kapranov-Zelevinsky~\cite{Gelfand1994} and to mixed discriminants in
the sense of~\cite{Cattani2013}. Later, in Section~\ref{sec: discriminants}, we extend this viewpoint to
systems whose equations are linear combinations of fixed prescribed
polynomials.
\begin{example}
Generically, the following  system has two isolated solutions in the torus~$(\KK^*)^2$:
\begin{equation}\label{eq: sysIntro}
    a_0 + a_1x + a_2y + a_3(x^2+y^2) \;=\;
b_0 + b_1x + b_2y + b_3(x^2 +y^2) \;=\; 0.
\end{equation}
These solutions collide precisely when the point
$
([a_0:a_1:a_2:a_3],[b_0:b_1:b_2:b_3])$ in $(\PP^3)^\vee \times (\PP^3)^\vee
$
lies on a hypersurface $\Delta_\mathrm{mixed}$ of degree $4$, defined by a polynomial with $14$ terms:
\begin{equation}\label{eq: discIntro}
 \Delta_\mathrm{mixed} =   
 4a_1a_3b_0b_1 - 4a_3^2b_0^2 + \dots 
 +4a_0a_2b_2b_3-4a_0^2b_3^2 = \det \begin{pmatrix}
     Q(a) & B \\
     B & Q(b)
 \end{pmatrix},
\end{equation}
where $Q(x) = x_1^2 + x_2^2 - 4x_0x_3$ and $B = a_1b_1 + a_2b_2 - 2a_0b_3 - 2a_3b_0$.
\hfill~$\diamond$  
\end{example}
\noindent
This simple example already shows the geometry behind the discriminant. Consider  the map
$$
\varphi\colon (\mathbb C^*)^2 \to \mathbb P^3, \quad (x,y)\mapsto
[1:x:y:(x^2+y^2)].
$$
Let $X\subset \mathbb P^3$ be the closure of its image. The two
coefficient vectors $a,b\in(\mathbb P^3)^\vee$ define hyperplanes
$H_a,H_b$, hence a line $L_{ab}=H_a\cap H_b\subset\mathbb P^3$.
The solutions of the system are the preimages of
$X\cap L_{ab}$ in $(\mathbb C^*)^2$. A
multiple root occurs exactly when $L_{ab}$ meets $X$ non-transversely.
Thus the discriminant is the Hurwitz form of $X$~\cite{sturmfels2016hurwitzformprojectivevariety} in Stiefel coordinates~$a,b$.

This is part of a general dictionary. Given a projective variety
$X\subset \mathbb P^n$, its higher associated varieties parametrize
$\ell$-linear spaces in the corresponding Grassmannians ${\rm Gr}(\ell,\PP^n)$ whose intersection with
$X$ is non-transverse. This also includes the Chow hypersurface~\cite{Dalbec1995}  and the
projective dual variety $X^\vee\subset(\mathbb P^n)^\vee$. By the
celebrated Cayley trick, these higher associated varieties can be
studied as projective dual varieties of Segre embeddings. Thus the study
of the discriminants as above, equivalently, higher associated hypersurfaces,
reduces to the study of dual hypersurfaces, or of the corresponding incidence
varieties, namely conormal~varieties
$$
        {\cal N}_X =
        \overline{\{(x,H)\mid x\in X^{\mathrm{sm}},\ T_xX\subset H\}}
        \subset \mathbb P^n\times(\mathbb P^n)^\vee .
$$
The goal of this paper is to understand what happens to these objects under
flat degenerations.   We focus specifically on Gröbner~degenerations. More precisely, given a 
Gröbner family  with respect to  $\omega$ with irreducible general fiber
$X_t\subset \mathbb P^n$, we study the limit of the dual varieties
$X_t^\vee$, and more generally of the higher associated hypersurfaces. This problem is natural from the viewpoint of polynomial
systems.  Such flat degenerations are widely used to replace a given
polynomial system by a simpler initial system, often sparse, for which we know the solutions, and then deform its solutions back to the original system; see
\cite{BerndPolyhedralHomotopy, HuberSottileSturmfels1998, BurrSottileWalker2023, BettiBorovik2025}. This leads to the question of how discriminants behave under induced degenerations.
\begin{example}
Take a weight $\omega = (2,1)$. Consider the family of parametrization maps
$$
\varphi_t:(\mathbb C^*)^2\to\mathbb P^3,\quad
(x,y)\mapsto [1:x:y:x^2+t^2 y^2].
$$
The corresponding varieties $X_t = \overline{\rm{im}(\varphi_t)}$
degenerate to the normal toric variety $X_0=\overline{\operatorname{im}\varphi_0}$. On the level of polynomial systems, we deform the system~\eqref{eq: sysIntro} to the sparse system
\begin{equation}\label{eq: sparse intro}
 a_0+a_1x+a_2y+a_3x^2
=
b_0+b_1x+b_2y+b_3x^2
=0.   
\end{equation}
Its discriminant is the initial form of 
$\Delta_{\rm mixed}$~\eqref{eq: discIntro}, that can be also written as a determinant:
\begin{equation*}
    {\rm in}_v(\Delta_{\rm mixed}) = \det \begin{pmatrix}
        a_1b_2 - a_2b_1 & 2(a_2b_0-a_0b_2)\\
        2(a_2b_3-a_3b_2) & a_1b_2 - a_2b_1 
    \end{pmatrix}, \quad v = -(0,1,0,2,0,1,0,2).
\end{equation*}
That is, in this case the limit of the discriminant~\eqref{eq: discIntro}  is the discriminant of the limit~\eqref{eq: sparse intro}.
\hfill$\diamond$
\end{example}
\noindent
In the example above, the discriminants are Hurwitz forms, and the result follows from the regularity of the special
fiber $X_0$ in codimension one. This is our Corollary~\ref{cor: chhurdeg}, which also follows from~\cite[Theorem~4.1]{sturmfels2016hurwitzformprojectivevariety}. In general,
for dual varieties, other higher associated hypersurfaces or more general mixed discriminants, extra factors
and multiplicities may~appear. 

\smallskip

The  goal of this work is to describe these extra
components and their multiplicities systematically. Our approach builds on classical work of Lê, Teissier, Kleiman and Sabbah.
The induced degeneration of conormal varieties is in fact the relative conormal
space, whose relation to limits of tangent spaces was developed in
\cite{Teissier1982, LeTeissier1988}. Specializations of conormal schemes
in flat families were already studied in~\cite{Kleiman}, where it was shown  that extra components may appear over the singular
locus of the special fiber. Sabbah computed the
multiplicities in the limiting conormal cycle in terms of local Euler
obstructions and Milnor fibers~\cite{Sabbah1985}. In Section~\ref{sec: dualizing flat}, we adapt this
theory to our setting and make it explicit for Gröbner
degenerations: Proposition~\ref{prop: conormal deg} shows that the
conormal variety $\mathcal{N}_X$ admits a flat Gröbner degeneration with weight $(\omega,-\omega)$,
Theorem~\ref{thm: limit conormal strata} identifies the components of its special
fiber using certain Whitney stratifications, and Theorem~\ref{thm: sabbah mult} gives their
multiplicities using Sabbah's formula. Pushing forward to the dual
projective space yields 
factorization of the limiting
dual hypersurface in Proposition~\ref{prop: multiplicities in dual}. Thus, the results of Section~\ref{sec: dualizing flat} can be viewed as an
adaptation of the classical theory to the setting most
relevant for  applications.
\subsection*{Outline}
The rest of the paper is organized as follows. Section~\ref{sec: Preliminaries}
recalls the necessary background on Gröbner degenerations, projective
duality, higher associated hypersurfaces, and Whitney stratifications.
Section~\ref{sec: dualizing flat} gives the precise statements and proofs
for the limit of conormal varieties and dual hypersurfaces discussed
above. In Section~\ref{sec: dializing special}, we apply these results to
several explicit families: hypersurfaces with isolated singularities
(Proposition~\ref{prop: IHS}), generic
complete intersections (Lemma~\ref{lem: CI mult}), reciprocal linear
spaces (Theorem~\ref{thm: limitRLS}), and Gelfand-Tsetlin
degenerations of Grassmannians (Problem~\ref{prob: GT}). Section~\ref{sec: HighAssociated}
extends the theory to higher associated hypersurfaces in Grassmannians. Its
main result is Theorem~\ref{thm:higher-associated-limit}. Finally,
Section~\ref{sec: discriminants} applies the results to mixed
discriminants of horizontally parametrized polynomial~systems.

\noindent
All examples in this paper are accompanied by code, available in the open-access repository:
 \begin{center}
    \url{https://zenodo.org/records/21398038}.
\end{center}
\subsection*{Conventions and Notations}
We work over $\CC$. All schemes appearing in this paper are separated and of finite type, so we are in the setting of~\cite{EisenbudHarris2016}. By a variety we mean a reduced separated scheme of finite type.  The notation
 $\mathcal X/\Ab^1$ means a flat family
$\pi:\mathcal X\to\Ab^1$. Its total space $\mathcal X$ is assumed to be reduced and
irreducible.
For a pure-dimensional scheme $Z$, we write $[Z]$ for the 
\emph{fundamental
 algebraic cycle}. If
$
Z_1,\ldots,Z_s
$
are the irreducible components of~$Z^{\rm red}$ with general points $\eta_1,\ldots,\eta_s$, then the cycle $[Z]$ has \emph{geometric multiplicities} of $Z_i$ in $Z$ as coefficients, see~\cite[Sec.~1.5,4.3]{Fulton1998}:
\begin{equation}\label{eq: alg cycle}
[Z]
=
\sum_{i=1}^s
\operatorname{mult}_{Z_i}(Z)[Z_i],
\quad
\operatorname{mult}_{Z_i}(Z)
=
\operatorname{length}_{\mathcal O_{Z,\eta_i}}\mathcal O_{Z_i,\eta_i}.
\end{equation}
By the limit of a degeneration we 
mean the \emph{flat limit}, see~\cite{VakilRisingSea}. That is, it is the special fiber of the scheme-theoretic closure of $\pi^{-1}(\Ab^1\setminus\{0\})$. Its
support is the set-theoretic limit.

\section{Preliminaries}\label{sec: Preliminaries}
In this section we recall only the most important notions that will be used throughout
the paper. Additional background and notation will be introduced directly in the relevant sections.

\subsection*{Flat Gr\"obner Degenerations}

We begin with basic notions of flat and Gr\"obner degenerations
that will be used in the paper.
In this work, by a one-parameter \emph{flat degeneration} of a variety $X$ we mean a flat morphism
$\pi\colon \mathcal{X} \to \KK$ whose fibers over $\KK^\times$ are all
isomorphic to $X$ (called the \emph{general fibers}), and whose fiber over~$0$ is
the \emph{special fiber}. When the special fiber is a toric variety, the
degeneration is called a \emph{toric degeneration}. Flatness ensures that
the fibers share invariants determined by the Hilbert function, such as
dimension and degree.
In particular, we focus on one-parameter flat families determined by weight orders on the 
coordinate ring. These are   \emph{Gröbner degenerations}.  By \cite[Theorem~15.17]{Eisenbud}, Gr\"obner degenerations are flat. We briefly recall the construction. 

\smallskip

For an ideal
$I \subset \KK[y]$ with a degree-compatible monomial order $\succ$,
its initial ideal  
$
\mathrm{in}_{\succ} (I) = \mathrm{span}\{\mathrm{in}_{\succ}(g) \mid g \in I\}
$
is finitely generated by Hilbert's Basis Theorem, and any set of polynomials $G\subset I$
whose leading terms generate $\mathrm{in}_{\succ} (I)$ forms a Gr\"obner basis
of $I$. In what follows we use the following generalization of monomial orderings: let $\omega \in \mathbb{Z}^{n+1}$ be a \emph{weight vector}. For 
$f=\sum_\alpha c_\alpha y^\alpha \in \KK[y]$, the initial form
$\operatorname{in}_\omega(f)$ is the sum of those terms for which $\omega\cdot\alpha$
is maximal. In particular, $\operatorname{in}_\omega(f)$ need not be a monomial and, by~\cite[Prop.~1.11]{sturmfels1996grobner}, for any ideal $I$ one can choose $\omega$ such that $${\rm in}_\omega (I) = {\rm in}_\succ (I).$$ 
Gr\"obner bases naturally give rise to deformations of varieties.
Indeed, 
$\KK^\times$ acts on $\KK[y]$~by
$$
t \cdot y^\alpha \coloneqq t^{-\omega \cdot \alpha}y^\alpha.
$$
For a polynomial $f=\sum c_\alpha y^\alpha \in \KK[y]$, set
$v(f)\coloneqq\max_\alpha\{\langle \omega,\alpha \rangle\}$ and define the \emph{weighted homogenization} of the polynomial $f$ and the ideal $I$ with respect to a weight $\omega$ as:
$$
f_\omega^t \coloneqq t^{v(f)}(t\cdot f), \qquad
I_\omega^t \coloneqq (f_\omega^t \mid f\in I) \subseteq \KK[y,t].
$$
The resulting family of ideals is called the \emph{Gr\"obner degeneration}
of $I$ with respect to~$\omega$. We will also often use this term to refer to the family of projective varieties defined by these~ideals.

\subsection*{Gelfand-Kapranov-Zelevinsky Setup}

We start by recalling the fundamentals on projective duality. Namely, in this section we define the conormal and dual varieties of an irreducible projective variety, as well as its higher associated hypersurfaces in Grassmannians. For a more detailed exposition, see~\cite{Gelfand1994}.


\smallskip

Let $V$ be a vector space of dimension $n+1$ and denote by $\PP(V)$ its projectivization. We identify $\PP(V) \cong \PP(\KK^{n+1}) \eqcolon \PP^n$ by picking a basis of $V$.
The dual projective space ${(\PP^n)^\vee \coloneq \PP(V^\vee)}$ parametrizes hyperplanes in $\PP^n$: a point $z\in\PP(V^\vee)$ corresponds to the~hyperplane defined by
$$
H_z=\Bigl\{y\in\PP^n \ \Big|\ \sum_{i=0}^n z_i y_i=0\Bigr\}.
$$
Let $X\subset\PP^n$ be an irreducible projective variety 
of dimension $d$ and codimension $c \coloneq n-d$
with homogeneous vanishing ideal $I_X\subseteq \KK[y_0,\dots,y_n]$ and smooth locus denoted by
$X^\mathrm{sm}$. For a point $x\in X^\mathrm{sm}$,
the $d$-dimensional \emph{embedded projective tangent space} of $X$ at $x$
is given by
$$
T_xX
=
\Bigl\{y\in\PP^n\ \Big|\
\sum_{i=0}^n \frac{\partial g}{\partial y_i}(x)\,y_i=0
\ \text{for all } g\in I_X\Bigr\}.
$$

\begin{definition}
The \emph{conormal variety} of $X$ is the 
closure $
\mathcal{N}_X
=
\overline{\mathrm{Con}(X^{\rm sm})}$
in
$\PP^n\times(\PP^n)^\vee
$
of 
\begin{equation}\label{eq: conormal bundle}
 \mathrm{Con}(X^{\rm sm})
=
\bigl\{(x,z)\in \PP^n\times(\PP^n)^\vee
\mid x\in X^\mathrm{sm},\ T_xX\subseteq H_z\bigr\}.
\end{equation}
It is an irreducible projective variety of dimension $n-1$.
\end{definition}
\begin{remark}\label{rem: conormal ideal} We can explicitly write the defining ideal of $\mathcal{N}_X$. Indeed, 
let $I_X=(g_1,\dots,g_r)$ be the vanishing ideal and $J=(\partial g_i/\partial y_j)_{1 \leq i \leq r, 0 \leq j \leq n}$ be the $r\times(n+1)$ Jacobian. Consider
$$
M\coloneq\begin{pmatrix} z_0 & \cdots & z_n \\ & J & \end{pmatrix},
$$
called the \emph{augmented Jacobian matrix}.
Then the conormal variety $\mathcal{N}_X$ is cut out by the defining equations of $X$ together with the $(c+1)\times(c+1)$ minors of $M$, after removing the components supported on  $\rm{Sing}(X)$ and on $\{z=0\}$ by saturation. In symbols, we have
$$
I_{\mathcal{N}_X}
=
\sqrt{\Bigl(\bigl(I_X+I_{c+1}(M)\bigr):\bigl(I_X+I_c(J)\bigr)^\infty\Bigr):(z_0,\dots,z_n)^\infty}.
$$
\end{remark}
\noindent
The conormal variety is equipped with two natural  projections on $\PP^n$ and $(\PP^n)^\vee$, respectively:
$$
\pr_1\colon\mathrm{Con}(X^{\rm sm})\to\PP^n, \quad \pr_2\colon\mathrm{Con}(X^{\rm sm})\to(\PP^n)^\vee.
$$
Both of them encode information about $X$.
While the first projection $\pr_1$ recovers $X^\mathrm{sm}$,
the image of the second projection $\pr_2$ parametrizes
hyperplanes tangent to $X$ at smooth points. The closure of the latter image  is called the \emph{projectively dual variety} of $X$, denoted by $X^\vee$. 

\begin{definition}
The \emph{projectively dual variety} of $X$ is defined as
$
{X^\vee \coloneqq \overline{\pr_2(\mathrm{Con}(X^{\rm sm }))}}
$.
\end{definition}

\noindent
When $X^\vee$ is a hypersurface, its defining polynomial is called the \emph{$X$-discriminant}, denoted by
$$
\Delta_X\in \KK[z_0,\dots,z_n].
$$
If $\codim(X^\vee)>1$, we call $X$
\emph{defective} and set $\Delta_X=1$. 
Returning to the conormal variety, we list some of its most important properties.
A fundamental result is \emph{biduality}: 
$\mathcal{N}_X = \mathcal{N}_{X^\vee}$.
Recall now that the cohomology ring of
$\PP^n\times(\PP^n)^\vee$ is a truncated polynomial ring
$$
H^*(\PP^n\times(\PP^n)^\vee,\Z)
\simeq
\Z[s,t]/(s^{n+1},t^{n+1}),
$$
where $s$ and $t$ denote the hyperplane classes pulled back from the first
and second factors, respectively. The cohomology class of $\mathcal{N}_X$ in $H^*(\PP^n\times(\PP^n)^\vee,\Z)$ is given by the bilinear form
$$
[\mathcal{N}_X]
=
\delta_1(X)\,s^n t
+
\delta_2(X)\,s^{n-1}t^2
+
\cdots
+
\delta_n(X)\,s\, t^n.
$$
The integers $\delta_i(X)$ are called the \emph{polar degrees} of $X$. They are positive if and only if
$$\codim(X^\vee) \le i\le \dim(X)+1.$$ Moreover,
$\delta_{\dim(X)+1} = \deg(X)$ and $\delta_{\codim(X^\vee)} = \deg(X^\vee)$;
for further properties see~\cite{KOHN2021157}.

\medskip

Finally, we recall the notion of \emph{higher associated hypersurfaces} of $X$. Let
$\Gr(\ell,\PP^n)$ denote the Grassmannian parametrizing linear spaces of projective dimension $\ell$  in
$\PP^n$. For such linear space $L \in \Gr(\ell,\PP^n)$, take a matrix
$A \in \KK^{(n-\ell)\times(n+1)}$  such that 
\[
L = \PP(\ker(A)).
\]
The entries of $A$ are called the
\emph{Stiefel coordinates} of $L$, and the maximal minors
$p_{i_1\cdots i_{n-\ell}}$ of $A$ are its \emph{Pl\"{u}cker coordinates}.
Note that the Pl\"{u}cker coordinates are uniquely determined up to a common scalar
factor, whereas the Stiefel coordinates are not. Indeed, left multiplication of~$A$ by any invertible
$(n-\ell)\times(n-\ell)$ matrix does not change its kernel.

\smallskip

The dual variety $X^\vee$ records hyperplanes tangent to
$X^{\mathrm{sm}}$, equivalently, hyperplanes whose intersection with
$X$ is non-transverse. This idea extends to linear spaces of
arbitrary~dimension.
\begin{definition}\label{def: assocHyper}
For $i\in\{0,\dots,d\}$, define  the \emph{higher associated variety} of $X$ as
\begin{equation}\label{eq: assocHyper}
\mathcal{Z}_i(X)
\coloneqq
\overline{
\bigl\{
L\in \Gr(n-d+i-1,\PP^n)
\mid \exists\,x\in X^\mathrm{sm} \cap L:\ \dim(L\cap T_xX)\ge i
\bigr\}
}.    
\end{equation} 
\end{definition}
\noindent
Notice that for $i=d$, the variety $\mathcal{Z}_d(X)$ is exactly  the projective
dual variety $X^\vee\subset(\PP^n)^\vee$.
We will focus on the case when $\mathcal{Z}_i(X)$ is
a hypersurface in $\Gr(n-d+i-1,\PP^n)$. This happens precisely when
$i\le\dim(X)-\codim(X^\vee)+1$.
For $i=0$, $\mathcal{Z}_0(X)$ is the \emph{Chow~hypersurface} of $X$,
consisting of linear spaces $L\in\Gr(n-d-1,\PP^n)$ that meet $X$. Its
defining equation in Pl\"{u}cker coordinates has degree $\deg X$ and is
called the \emph{Chow form} of $X$. For $i=1$,~$\mathcal{Z}_1(X)$ is the
\emph{Hurwitz variety}, parametrizing linear spaces
$L\in\Gr(n-d,\PP^n)$ that meet $X$ non-transversely. When $\deg X > 1$, it
is a hypersurface and its defining equation is called the \emph{Hurwitz
form}~of~$X$.

\subsection*{Whitney Stratifications}
Whitney stratifications will be one of the main tools we will use in what follows. We introduce them briefly in this section. For more details, we refer to the textbooks~\cite{Dimca1992, GoreskyMacPherson1988}.

\smallskip

A \emph{stratification} of 
$X$ is a finite partition
$\mathcal S=\{S_\alpha\}_{\alpha \in A}$ of $X$ into smooth locally closed subvarieties, called
\emph{strata}, such that for every $\alpha \in A$, the boundary
$\overline{S_\alpha}\setminus S_\alpha$ is a union~of~strata.

\begin{definition}\label{def: WhitStrat}
A stratification $\mathcal S$ is called a  \emph{Whitney stratification} if for every
pair of strata satisfying $S_\beta\subset\overline{S_\alpha}\setminus S_\alpha$ the following conditions hold:
\begin{enumerate}
    \item[\rm(a)] for every sequence $(x_n)\subset S_\alpha$ with
    $x_n\to x\in S_\beta$ and $T_{x_n}S_\alpha\to T$, one has
    $T_xS_\beta\subseteq T$;
    \item[\rm(b)] for every pair of sequences
    $(x_n)\subset S_\alpha$, $(y_n)\subset S_\beta$ with
    $x_n,y_n\to x\in S_\beta$, if the secant lines
    $\overline{x_ny_n}$ converge to a line $\ell$ and the tangent planes 
    $T_{x_n}S_\alpha$ converge to $T$, then $\ell\subseteq T$.
\end{enumerate}
\end{definition}
\noindent
A Whitney stratification is called \emph{minimal} (or \emph{coarsest}) if, after merging any number of
strata, the resulting stratification no longer satisfies the above conditions. A fundamental theorem is that every complex algebraic variety admits a unique minimal Whitney
stratification, see~\cite{Whitney1965, Teissier1982, Mather2012}. When a closed subvariety $Y\subset X$ is fixed, by a minimal Whitney
stratification \emph{compatible with $Y$} we mean a coarsest Whitney stratification
of $X$ refining the partition
$
X=(X\setminus Y)\sqcup Y.
$
Equivalently, the subvariety $Y$ is required to be a union of strata. 
\begin{definition}\label{def: map strat}
A Whitney stratification of a morphism $f:X\to Y$ is a pair of Whitney
stratifications $\mathcal S_X$ of $X$ and $\mathcal T_Y$ of $Y$ such
that for every stratum $S\in\mathcal S_X$ there exists a stratum
$T\in\mathcal T_Y$ such that
$
f(S)\subseteq T
$
and the restriction
$
f|_S:S\to T
$
is a submersion onto its image.    
\end{definition}
Whitney stratifications of maps always exist, see Theorem~1 in Section~4
of~\cite{Hir77}.~For checking Whitney regularity, the L\^e-Teissier conormal
criterion~\cite{LeTeissier1988} provides an algebraic test. Effective
algorithms for computing Whitney stratifications are given in
\cite{HelmerNanda2022,HelmerMohr2024}.
\begin{remark}\label{rk: GAGApple}
We note that Whitney stratifications are more common in the study of complex analytic
spaces, that is, complex spaces where singularities are allowed. These are
more general than complex algebraic varieties, since they are locally
defined by holomorphic functions rather than polynomials. They are naturally
endowed with the Euclidean topology.

However, throughout, we identify complex algebraic varieties with their associated
ana\-lytic spaces whenever we discuss limits of points or tangent spaces, and
Whitney conditions. Indeed, for constructible subsets of a complex algebraic variety, Euclidean and Zariski
closure have the same underlying set. Therefore, points in the Zariski closure of a constructible set may be
approached by sequences of complex points from that set. 
Conversely, by Chow's theorem~\cite{Chow1949}, closed
analytic subsets of a complex projective variety are also Zariski~closed.
\end{remark}
\begin{remark}\label{rem: whitney properties}
We will use two standard properties of Whitney stratifications.
First, if $\mathcal S=\{S_\alpha\}$ and
$\mathcal T=\{T_\beta\}$ are some Whitney stratifications of $X$  and $Y$, respectively, then
$$
\mathcal S\times\mathcal T
=
\{S_\alpha\times T_\beta\}_{\alpha,\beta}
$$
is a Whitney stratification of $X\times Y$. In particular, if $Y$ is smooth,
then $\{S_\alpha\times Y\}_\alpha$ is a Whitney stratification of
$X\times Y$. Second, if $Z\subset X$ is a smooth subvariety transverse to
all strata $S_\alpha$, then the nonempty intersections
$
S_\alpha\cap Z
$
form a Whitney stratification of $X\cap Z$. These facts are direct
consequences of the definition; see, for
example,~\cite[pp.~12-13]{Gibson1976}.    
\end{remark}

\section{Dualizing Flat Families}\label{sec: dualizing flat}

In this section we study the
behavior of projective duality under flat degenerations of~${X \subset \PP^{n}}$.


\smallskip


For the rest of this section, take 
$\Ab^1=\Spec \, \KK[t]$ and let
$
\XX\subseteq \PP^n\times \Ab^1
$
be a reduced, irreducible closed subscheme, flat over $\Ab^1$ with the projection denoted by 
$
\pi:\XX\rightarrow \Ab^1
$.   We write
$\mathcal I\subseteq \KK[t,y]$ for its vanishing ideal, homogeneous in
$ y=(y_0,\ldots,y_n)$, and 
$
X_t:=\pi^{-1}(t)
$
for the fiber over $t\in \Ab^1$. Additionally, we assume that  general fibers $X_t$ for $t\neq 0$ are reduced.

The degeneration of~$X$ induces a degeneration of the conormal variety $\mathcal{N}_{X}$. The latter construction has already
appeared in greater generality in the analytic setting under the~name of the
\emph{relative conormal space} of a map, see~\cite[Section 1.6]{trang2023limits}, and also~{\cite{Teissier1982,  LeTeissier1988,gaffney2019}}.
\begin{definition}\label{def: rel conormal}
The \emph{conormal space $\NN_{\XX/\mathbb{A}^1}$ relative to} $\pi$ is the scheme-theoretic closure  of the set of triples $(x, H, \pi(x))$
such that $x$ is a smooth point of the general fiber $\pi^{-1}(\pi(x))$
 and $H$ is a tangent
hyperplane to $\pi^{-1}(\pi(x))$ at $x$. In symbols, we have
\begin{equation}\label{eq: relative conormal}
 \NN_{\XX/\mathbb{A}^1}
:=
\overline{\bigcup_{t\in\mathbb{A}^1\setminus\{0\}}\mathrm{Con}(X_t^{\rm sm})}
\subseteq \PP^n\times(\PP^n)^\vee\times\mathbb{A}^1.   
\end{equation}
\end{definition}
\begin{remark}\label{rem: relative conormal ideal}
Similarly to Remark~\ref{rem: conormal ideal}, we can write down the defining ideal of $
\NN_{\XX/\mathbb{A}^1}$. 
Let~$J_y$ be the Jacobian matrix of the ideal $\mathcal{I}$ with
respect to the $y$-variables and consider the augmented Jacobian   matrix $M_y=\binom z{J_y}$. Then, 
$$
I_{\NN_{\XX/\mathbb{A}^1}}
=
\sqrt{\Bigl(\bigl(\mathcal I+I_{c+1}(M_y)\bigr):
\bigl(\mathcal I+I_c(J_y)\bigr)^\infty\Bigr):(z_0,\dots,z_n)^\infty},
$$
where $c = {\rm codim} (X_t)$ for $t \in \Ab^1$. Note that, by Remark~\ref{rem: no saturation by t}, no saturation by $(t)$ is needed.
\end{remark}
\subsection*{The support of the limiting conormal cycle}
In what follows, 
Proposition~\ref{prop: conormal deg} and
Theorem~\ref{thm: limit conormal strata} overlap with results in~\cite{Kleiman}. Kleiman showed that under a flat
degeneration the conormal scheme specializes to a union of conormal schemes,
with possible extra components lying over subvarieties of the singular locus
of the special fiber. Here we refine the description of these extra
components using Whitney stratifications, and we make the conormal degeneration
explicit for Gr\"obner degenerations.
\begin{proposition}\label{prop: conormal deg}
Let $\XX\to\Ab^1$ be the Gröbner degeneration of
$X\subset\PP^n$ with respect to the weight $\omega$. Then the relative conormal
space ${\cal N}_{\XX/\Ab^1}$ is the Gröbner degeneration of the conormal variety
${\cal N}_X\subset\PP^n\times(\PP^n)^\vee$ with respect to the weight $(\omega,-\omega)$.
In particular,
$$
{\cal N}_{X_t} = ({\cal N}_{\XX/\Ab^1})_t
\quad\text{for every }t\in\KK^\times,
$$
and, set-theoretically, we have the inclusion ${\cal N}_{X_0} \subseteq ({\cal N}_{\XX/\Ab^1})_0$.
\end{proposition}
\begin{proof}
For $t\in\KK^\times$, consider the following torus action on $\PP^n\times(\PP^n)^\vee$:
$$
\phi_t: 
(x,z)\mapsto
(t^\omega\cdot x,t^{-\omega}\cdot z) \coloneq \bigl(\,
[t^{\omega_0}x_0:\cdots:t^{\omega_n}x_n], \,
[t^{-\omega_0}z_0:\cdots:t^{-\omega_n}z_n]\,
\bigr).
$$
We show that $\phi_t({\rm Con}(X^{\rm sm}))={\rm Con}(X_t^{\rm sm})$. Indeed,
consider a Gröbner basis $\{g_1,\ldots,g_r\}$ of the vanishing ideal $I_X$ with respect to the weight $\omega$. Write
$
g_\omega^t(y)=t^{v(g)}g(t^{-\omega}\cdot y)$,
where $v(g_i) =\max_{\alpha \in \text{supp}(g_i)}\{\langle \omega,\alpha\rangle\}
$.
The point $t^\omega\cdot x$ is in $X_t$ if and only if $x\in X$, since
$$
(g_i)_\omega^t(t^\omega\cdot x)= t^{v(g_i)} g_i(t^{-\omega}\cdot t^\omega \cdot x) = t^{v(g_i)}g_i(x).
$$
Moreover,  the Jacobian matrix of $I_{X_t}$ at $t^\omega\cdot x$ is obtained
from the Jacobian matrix of $I_X$ at~$x$ by multiplying rows and
columns by nonzero scalars. Thus both matrices have the same rank:
\begin{equation}\label{eq: scaled jac}
\frac{\partial (g_i)_\omega^t}{\partial x_j}(t^\omega\cdot x)
=
t^{v(g_i)-\omega_j}
\frac{\partial g_i}{\partial x_j}(x).    
\end{equation}
In particular, smooth points of $X$ and of $X_t$ are in bijection under
$x\mapsto t^\omega\cdot x$.
Now let $(x,z)\in {\rm Con}(X^{\rm sm})$. By definition,
$T_{x}X \subset H_{z}$. Equivalently, 
$z$ lies in the row span of the Jacobian matrix of $I_X$ at
$x$. Then, by~\eqref{eq: scaled jac}, the point 
$t^{-\omega}\cdot z$ lies in the row span of the Jacobian matrix of
$X_t$ at $t^\omega\cdot x$. Hence, we have
$$T_{t^\omega\cdot x} (X_t) \subset H_{t^{-\omega} \cdot z} \quad \text{and} \quad 
(t^\omega\cdot x,t^{-\omega}\cdot z)\in {\rm Con}(X_t^{\rm sm}).$$
The same argument applied to the inverse rescaling gives the reverse
inclusion. We obtain
\begin{equation*}
{\rm Con}(X_t^{\rm sm})     = (t^\omega, t^{-\omega}) \cdot {\rm Con}(X^{\rm sm}) \quad \forall \, t\neq 0.
\end{equation*}
By definition, the closure of the LHS over all $t\in \CC^\times$ is exactly the relative conormal~${\cal N}_{\mathcal X/\Ab^1}$, while the closure of the RHS gives
the Gröbner degeneration of ${\cal N}_X$ with respect to the weight vector~$(\omega,-\omega)$.
Consequently, we have the equality 
$
({\cal N}_{\mathcal X/\Ab^1})_t={\cal N}_{X_t}
$ for every $t\in\KK^\times$.

\smallskip

To show the last statement, take $(x_0,H_0)\in \operatorname{Con}(X_0^{\rm sm})$. By definition, 
$x_0\in X_0^{\rm sm}$ and
$
T_{x_0}X_0\subset H_0
$. We show that it is a limit of conormal points of the nearby fibers.
Indeed, let $J(t,x)$ be the Jacobian matrix of
the vanishing ideal of the fiber $X_t$, with respect to the fiber variables, as in Remark~\ref{rem: relative conormal ideal}.
Since $x_0\in
X_0^{\rm sm}$, $J(0,x_0)$ has the expected rank $\codim(X_0)$. Hence, we may choose a curve
$
x(t)\in X_t^{\rm sm}$ with $x(0)=x_0
$.
Let $z_0$ define $H_0$. The condition
$T_{x_0}X_0\subset H_0$ means that $z_0$ lies in the row span of
$J(0,x_0)$, so choose $\lambda$ with
$
z_0=\lambda^\top J(0,x_0).
$
Set $z(t)=\lambda^\top J(t,x(t))$. Then $z(t)\to z_0$, and
$z(t)$ lies in the row span of $J(t,x(t))$, hence
$$
T_{x(t)}X_t\subset H_{z(t)} \; \iff \; (x(t),H_{z(t)})\in \operatorname{Con}(X_t^{\rm sm})
\quad\text{and}\quad
(x(t),H_{z(t)})\to (x_0,H_0).
$$
Therefore
$
\operatorname{Con}(X_0^{\rm sm})
\subseteq
(\mathcal N_{\mathcal X/\mathbb A^1})_0,
$
and taking closures gives inclusion
$
\mathcal N_{X_0}\subseteq
(\mathcal N_{\mathcal X/\mathbb A^1})_0
$.
\end{proof}
\begin{remark}\label{rem: no saturation by t}
In the proof we showed that
$
\operatorname{Con}(X_0^{\rm sm})\subseteq
(\mathcal N_{\mathcal X/\Ab^1})_0
$.
Hence, 
the closure in
Definition~\ref{def: rel conormal} may equivalently be taken over the whole
line $\Ab^1$. Thus our definition agrees with the usual analytic definition
of the relative conormal space, as in~\cite{Teissier1982,  LeTeissier1988, gaffney2019}.
\end{remark}    
The above proposition immediately implies the following statement for dual varieties.
\begin{corollary}\label{cor: dual deg}
A flat Gr\"obner degeneration of $X\subset\PP^n$ with respect to the weight $\omega$ induces the flat 
Gr\"obner degeneration of $X^\vee\subset(\PP^n)^\vee$ with respect to the opposite weight
$-\omega$. Moreover, 
$$
(X_t)^\vee=(X^\vee)_t \quad \text{for } \; t\neq0,
\quad \text{and} \quad 
(X_0)^\vee\subseteq (X^\vee)_0.
$$
If $\dim (X_0)^\vee=\dim X^\vee$, then
$
\deg (X_0)^\vee\leq \deg X^\vee
$ with equality if and only if $(X_0)^\vee = (X^\vee)_0$.
\end{corollary}
\begin{remark}
Proposition \ref{prop: conormal deg} and Corollary \ref{cor: dual deg} were proven for the vanishing locus of the projective varieties involved.
Therefore, in the statements above, if the special fiber $X_0$ is non-reduced, then
$\mathcal N_{X_0}$ and $X_0^\vee$ are understood as
$
\mathcal N_{X_0^{\mathrm{red}}}$ and
$
(X_0^{\mathrm{red}})^\vee
$,
respectively. Moreover, if $X_0$ is reducible, ${\cal N}_{X_0}$  or $X_0^\vee$ is the union of the conormal or dual
varieties of the irreducible components of $(X_0)^{\mathrm{red}}$. The scheme structure of $X_0$, however, still affects the
multiplicities appearing in the limiting conormal cycle
$
[(\mathcal N_{\mathcal X/\mathbb A^1})_0]
$, as we will explain in the next subsection.
\end{remark}
The inclusions in Proposition~\ref{prop: conormal deg} and Corollary~\ref{cor: dual deg} are strict in general. The next example shows that the limit of the dual (and hence conormal) variety may have extra components. 
\begin{example}\label{ex: umbrella}
Consider the family
$
{\cal X}={\cal V}(y_1^2y_2-y_0^2y_3+t^2(y_3^3+y_2^3))\subset\PP^3\times \Ab^1 
$. 
This is the flat Gr\"obner degeneration of the smooth cubic with weight $\omega=(1,1,0,0)$.
The special fiber is
$$
X_0={\cal V}(y_1^2y_2-y_0^2y_3),
$$
known as the Whitney umbrella. Its dual $(X_0)^\vee$ is given by ${\cal V}(z_1^2z_2+z_0^2z_3)$,
whereas the special fiber of the induced dual degeneration with respect to the weight $(-1,-1,0,0)$ is 
\begin{equation}\label{eq: limit Whitney}
 (X^\vee)_0
=
{\cal V}\left(
z_3^3z_2^3(z_3^3-z_2^3)
(z_1^2z_2+z_0^2z_3)
\right).   
\end{equation}
Thus, we have strict inclusions:  $(X_0)^\vee\subsetneq (X^\vee)_0$ and $\mathcal{N}_{X_0} \subsetneq (\mathcal{N}_{\mathcal{X}/\Ab^1})_0$.
\hfill~$\diamond$  \end{example}
The extra factors in Example~\ref{ex: umbrella} come from subvarieties of the
singular locus of the special fiber $X_0$. We now introduce the tool which
will allow us to identify these subvarieties.
\begin{definition}\label{def: thom}[Thom's $a_\pi$-condition]
Let $\mathcal S$ be a stratification of the map
$\pi:{\XX\to\Ab^1}$. 
We say that
$\mathcal S$ satisfies \emph{Thom's $a_\pi$-condition} if, for every pair of strata satisfying $S_\beta\subset\overline{S_\alpha}\setminus S_\alpha$, every point $x\in S_\beta$,
and every sequence $(x_n)\in S_\alpha$ with $x_n\to x$, we~have
$$
\text{if } \; T_{x_n}\bigl(S_\alpha\cap X_{\pi(x_n)}\bigr)\to T, \quad \text{then } \; 
T_x\bigl(S_\beta\cap X_{\pi(x)}\bigr)\subseteq T.
$$
In particular, if $S_\alpha\subseteq X_0$, this gives the non-relative Whitney {\rm (a)}-condition~\eqref{def: WhitStrat}.
\end{definition}
The existence of Thom's $a_\pi$-stratifications for morphisms mapping to smooth curves
goes back to Hironaka~\cite[Sec.~5, Corollary~1]{Hir77}. 
In our setting, this condition simplifies considerably. 
Indeed, any stratification of the map $\pi$ in the sense of
Definition~\ref{def: map strat} will satisfy Thom's  $a_{\pi}$-condition. Let $\mathcal S$ be such a Whitney stratification
of $\pi:\XX\to\Ab^1$. 
Since $\pi$ is projective,
hence proper, Thom's first isotopy lemma~\cite{Thom1969} implies that the restriction
$
\pi|_S:S\to \pi(S)
$
is a locally trivial fibration for every stratum $S \in \mathcal{S}$. The theorem of
Briançon-Maisonobe-Merle~\cite[Theorem 4.2.1]{Briancon1994} then implies that $\mathcal S$ satisfies
Thom's~$a_\pi$-condition. 

Moreover, for Gröbner degenerations, it is enough to use the usual Whitney stratification (in the sense of Definition~\ref{def: WhitStrat})
of $\mathcal X$ compatible with $X_0$, as shown in the following lemma. 
\begin{lemma}\label{lem: Whitney for Grobner}
Let $\mathcal X\to \mathbb A^1$ be a Gröbner degeneration, and let
$\mathcal S$ be the minimal Whitney stratification of $\mathcal X$
compatible with the special fiber $X_0$. Then a Whitney stratification of the map $\pi:\mathcal X\to\mathbb A^1$
is given by $\mathcal S$ and the stratification
$\{\Ab^1\setminus\{0\},\{0\}\}$ of the line $\mathbb A^1$.
\end{lemma}
\begin{proof}
We have to show that for every stratum $S$ of $\mathcal S$, the map
$\pi|_S$ is a submersion onto the corresponding stratum of
$\{\mathbb A^1\setminus\{0\},\{0\}\}$. Since $\mathcal S$ is compatible with $X_0$, every stratum is either
contained in $X_0$ or disjoint from~$X_0$.
As in the proof of Proposition~\ref{prop: conormal deg}, one can show that a
Gröbner degeneration is equivariantly trivial over the punctured line
$\mathbb{A}^1 \setminus \{0\}$, meaning,
$
\mathcal X^\times:=\mathcal X\cap \pi^{-1}(\mathbb C^\times)
\simeq
X\times\mathbb C^\times
$.
Under this trivialization, the $\mathbb C^\times$-action is given by 
$$
\rho_\lambda: 
\mathcal{X} \to \mathcal{X},
\quad
(x,t)\mapsto (\lambda^\omega \cdot x,\lambda t),
\quad \lambda\in\mathbb C^\times.
$$
We claim that the automorphisms $\rho_\lambda$ preserve the restricted
stratification $\mathcal S|_{\mathcal X^\times}$. Indeed, viewing 
$\rho_\lambda$ as a diffeomorphism of the
underlying analytic space, the image of $\mathcal S$ is again a Whitney
stratification by~\cite[p.~11]{Gibson1976}. Since each $\rho_\lambda$ preserves the special fiber
$X_0$, the image stratification $\rho_\lambda(\mathcal{S})$ is again compatible with $X_0$. By
uniqueness of the minimal Whitney stratification compatible with $X_0$, the
action therefore permutes the strata of $\mathcal S|_\mathcal{X}^{\times}$. 

\smallskip

There are only
finitely many such strata and $\mathbb C^\times$ is connected, thus, each stratum $S\subset \mathcal X^\times$ is in
fact fixed, and therefore has the form
$
S=S_1\times \mathbb C^\times$, where
$S_1=S\cap (X\times\{1\})
$.
On such a stratum, the map $\pi$ is just the projection to $\mathbb C^\times$, therefore
$\pi|_S$ is a submersion.
The remaining strata are contained in $X_0$. They map to the point $\{0\}$, and the differential
to the tangent space of a point is automatically surjective. Thus $\pi$ is
a stratified submersion.
\end{proof}
\begin{remark}\label{rem: vertical strata}
A stratum of $\mathcal X$ will be called \emph{vertical} if it is contained
in a single fiber of~$\pi$. The proof of the lemma above shows that, for the chosen
minimal Whitney stratification of a Gr\"obner degeneration, the only vertical
strata are those contained in the special fiber $X_0$.
\end{remark}
The following result makes the source of the extra components precise: after
choosing a Whitney stratification of the map, compatible with the special
fiber, every irreducible component of the limiting conormal cycle comes from one of the
strata contained in the~special~fiber.
\begin{theorem}\label{thm: limit conormal strata}
Let $\XX\to\Ab^1$ be a flat Gröbner degeneration, and let
$\mathcal S$ be the minimal Whitney stratification of 
$\XX$ compatible with $X_0$. Then every irreducible component~$W$ of the reduced special fiber
$
(\mathcal N_{\XX/\Ab^1})_0^{\mathrm{red}}
$
is the conormal variety of a stratum contained in~$X_0$, i.e., 
$$\exists \, S\in\mathcal S, \quad \text{with }\; S\subseteq X_0, \quad  \text{such that } \;  W=\mathcal N_{\overline{S}}.$$
\end{theorem}
\begin{proof}
The maximal-dimensional strata of $X_0$ recover the ordinary conormal
variety $\mathcal N_{X_0}$, which is always 
in the limit by
Proposition~\ref{prop: conormal deg}. Thus, any extra component must come from a
lower-dimensional stratum. Take the conormal incidence over the \emph{relative smooth locus}
$$
S^\circ=
\{(x,H,t)\mid x\in (X_t)^{\rm sm},\ T_xX_t\subseteq H\}
\subset \PP^n\times(\PP^n)^\vee\times\Ab^1.
$$
By definition,
$
\mathcal N_{\XX/\Ab^1}=\overline{S^\circ}
$.
Let $W$ be an irreducible component of the special fiber 
$(\mathcal N_{\XX/\Ab^1})_0^{\rm red}$, and let
$
\eta=(x,H,0)\in W
$.
Consider the stratum  $S_\beta\subseteq X_0$ containing $x$. Since
$\eta\in\overline{S^\circ}$, by Remark \ref{rk: GAGApple} we may choose a sequence
$
(x_n,H_n,t_n)\in S^\circ
$
converging to $(x,H,0)$. Since the stratification is
finite, after passing to a subsequence we may assume that all $x_n$ lie in a
single stratum $S_\alpha$. Then
$
x\in\overline{S_\alpha}$, and thus
$S_\beta\subseteq\overline{S_\alpha}
$.
Along the latter sequence~we~have~the~inclusions
\begin{equation}\label{eq: relative tangent inclusion}
T_{x_n}\bigl(S_\alpha\cap X_{t_n}\bigr)
\subseteq
T_{x_n}X_{t_n}
\subseteq H_n.    
\end{equation}
After passing to a further subsequence, the tangent spaces
$T_{x_n}(S_\alpha\cap X_{t_n})$ converge in the corresponding Grassmannian to some linear space $T$. By Lemma~\ref{lem: Whitney for Grobner} together with the discussion preceding it, 
the chosen Whitney
stratification satisfies Thom's $a_\pi$-condition (in the sense of Definition \ref{def: thom}). Hence we have the inclusion
$
T_xS_\beta\subseteq T
$.
Passing to the limit in~\eqref{eq: relative tangent inclusion}
gives
$
T\subseteq H
$.
Therefore we have
$
T_xS_\beta\subseteq H
$
and hence $(x,H)$ lies in the component 
$\mathcal N_{\overline{S_\beta}}$.

\smallskip

We have shown that
$
W \subset \bigcup_{S\subseteq X_0}\mathcal N_{\overline S}
$. Since $W$ is irreducible, we have
$
W\subseteq \mathcal N_{\overline S}
$
for some stratum $S\subseteq X_0$. Finally, both varieties have dimension
$n-1$: the special fiber has 
dimension $n-1$ by flatness of the
conormal degeneration, and $\mathcal N_{\overline S}$ is the conormal
variety of an irreducible projective variety in $\PP^n$. Thus,
$
W=\mathcal N_{\overline S}
$, which completes the proof.
\end{proof}
\begin{remark}
The statement of Theorem~\ref{thm: limit conormal strata} is not special to
Gr\"obner degenerations. The same proof applies to any flat family
$\pi:\mathcal X\to\Ab^1$ of projective varieties, provided that one fixes a Whitney stratification of
the map $\pi$ (as in Definition~\ref{def: map strat}), compatible with~$X_0$. Notice that in this more general case, contrary to Lemma \ref{lem: Whitney for Grobner}, not every stratification of $\mathcal X$ compatible with $X_0$ will give a stratification of the map $\pi$.
\end{remark}
\begin{example}\label{ex: whitneyUmb revisit}
We revisit the degeneration of the Whitney umbrella from Example~\ref{ex: umbrella}. The minimal Whitney stratification of the family~${\cal X}$, compatible with $X_0$, was computed using the package \texttt{WhitneyStratifications} in \texttt{Macaulay2}. Its singular strata are given by
$$
\begin{aligned}
\dim 1:
& \quad 
L\coloneq \mathcal V(t, y_0,y_1),\\
\dim 0:
& \quad 
p_1 \coloneq \mathcal V(t, y_0,y_1,y_2), \;  p_2 \coloneq \mathcal V(t, y_0,y_1,y_3),\; \{q_1,q_2,q_3\} \coloneq \mathcal V(t, y_0,y_1,y_2^3+y_3^3).
\end{aligned}
$$
The one-dimensional stratum is the line 
$\operatorname{Sing}(X_0)=\mathcal V(t,y_0,y_1)$. The first two
points $p_1,p_2$ are the usual Whitney strata of the Whitney
umbrella. The last ideal defines three additional points $q_1,q_2,q_3$ on the singular line,
which arise as singularities of the total family ${\cal X}$ at $t=0$.

\smallskip

The line
$\mathcal V(t,y_0,y_1)$ is dual defective, and hence it does not contribute 
to the limiting dual~\eqref{eq: limit Whitney}. The five zero-dimensional strata contribute exactly the 
extra factors in~\eqref{eq: limit Whitney}: 
$$
p_1^\vee
=\mathcal V(z_3), \quad 
p_2^\vee
=\mathcal V(z_2),\quad 
\mathcal V(t,y_0,y_1,y_2^3+y_3^3)^\vee
=\mathcal V(z_3^3-z_2^3).
$$
Over $\CC$, one can factor the last cubic  into three linear forms defining $q_1^\vee,\; q_2^\vee,\; q_3^\vee$.
\hfill~$\diamond$  \end{example}
Theorem~\ref{thm: limit conormal strata} provides a finite list of candidates
for the irreducible components of the limiting conormal cycle. However,
not every such candidate has to occur. In other words, the converse to Theorem~\ref{thm: limit conormal strata} is false: not every
stratum $S\subseteq X_0$ of a Whitney stratification compatible with $\pi$ has to contribute a component ${\cal N}_{\overline{S}}$
to $(\mathcal{N}_{\mathcal{X}/\Ab^1})_0$, even if the stratification is chosen to be minimal. 
The following example, although reducible, illustrates this.
\begin{example}\label{ex: skew lines}
Consider $I_1 = (x,y)$ and $I_{2,t} = (z,y-tw)$.
Let
$
\XX={\cal V}\bigl(I_1 \cdot I_{2,t}\bigr)\subset \PP^3\times\Ab^1_t 
$
be a family with general fiber equal to the union of the two skew lines
$
L_1={\cal V}(I_1)
$
and
$
L_{2,t}={\cal V}(I_{2,t}),
$
and with special fiber equal to the union of the intersecting lines
$$
X_0=L_1\cup L_2,\quad
L_1={\cal V}\bigl(I_1\bigr), \quad
L_{2,0}={\cal V}\bigl(I_{2,0}\bigr),
$$
with a fat point $L_1\cap L_2=\{p\}$. The minimal Whitney stratification
of $\mathcal X$ restricted to~$X_0$~is
$$
L_1\setminus\{p\},\quad L_2\setminus\{p\},\quad \{p\}.
$$
However,
$
\bigl(\mathcal N_{\XX/\Ab^1}\bigr)_0
=
\mathcal N_{L_1}\cup\mathcal N_{L_2},
$
so the component $\mathcal N_{\{p\}}$ does not occur.
\hfill~$\diamond$  
\end{example}

In the preceding example the special fiber is non-reduced. This 
is not responsible for the failure of the converse to
Theorem~\ref{thm: limit conormal strata}: the same phenomenon can occur even
when the special fiber is reduced. The following example is adapted from
\cite[Remark~7.2]{GrigorievMilman2008}.
\begin{example}\label{ex: nonpersistent point}
Consider the flat family $\mathcal X=\mathcal V(F)\subset \PP^5_{x,y,u,v,w,z}\times\Ab^1_t$ of hypersurfaces with $$F=z(ux^2+vy^2)+2w^2xy+t z^4.$$
This is a Gröbner degeneration with respect to $\omega=(1,2,2,0,1,1)$. On the affine chart $z=1$, the restriction to $X_0$ of the minimal Whitney
stratification of the total family $\mathcal X$ has strata
$$
X_0\setminus L,\quad
L\setminus \mathcal V_L(\Delta),\quad
\mathcal V_L(\Delta)\setminus\{p\},\quad
\{p\}.
$$
Here,
$
L=\mathcal V(t,x,y)\simeq \Ab^3_{u,v,w}
$
is the singular locus of the special fiber $X_0$ on the affine~chart $z=1$, and
$
\Delta=uv-w^4
$
is the discriminant of the quadratic form
$
ux^2+2w^2xy+vy^2
$.
Finally, 
$$
p=[0:0:0:0:0:1]\in X_0
$$
is the singular point of the discriminant hypersurface
$\mathcal V_L(\Delta)\subset L$.
We now show that $\{p\}$ does not contribute to the limiting
conormal cycle. 
On the smooth locus of the fibers, conormal vectors on a hypersurface are given by
the gradient. Let $F_0 \coloneq F|_{t=0}$. On the chart $z=1$, 
$$
\nabla_{x,y,u,v,w} F_0
=
(2ux+2w^2y,\;2vy+2w^2x,\;x^2,\;y^2,\;4wxy).
$$
Let
$
M=\max_i|\partial_i F_0|
$.
Then
$
|4wxy|
\leq 2|w|(|x|^2+|y|^2)
=o(M),
$
since $M\geq \max\{|x|^2,|y|^2\}$. Moreover, on the nearby fibers of $\mathcal X$ we have
$t=-F_0$ on the chart $z=1$, and hence
$$
\partial_zF
=
ux^2+vy^2+4t
=
ux^2+vy^2-4F_0.
$$
We show that $|\partial_zF| = o(M)$. Indeed, by Euler's identity in the variables
$x$ and $y$, we have
$$
F_0
=
\frac12\left(
x\frac{\partial F_0}{\partial x}
+
y\frac{\partial F_0}{\partial y}
\right),
$$
and hence
$
|F_0|\leq \frac12(|x|+|y|)M=o(M)
$
as the point tends to $p$. Moreover,
$$
|ux^2+vy^2|
=
|F_0-2w^2xy|
\leq |F_0|+2|w|^2|xy| \leq |F_0| + |w|^2(|x|^2+|y|^2) = o(M),
$$
and consequently
$
|\partial_zF|
=o(M)
$. 
Therefore every limiting conormal vector over the point $p$ satisfies
$
\xi_w=\xi_z=0
$, where $\xi$ are the dual coordinates on $(\PP^5)^\vee$.
Thus, we have 
$$
\bigl(\mathcal N_{\mathcal X/\Ab^1}\bigr)_0
\cap
\bigl(\{p\}\times(\PP^5)^\vee\bigr)
\subseteq
\{p\}\times\mathcal V(\xi_w,\xi_z) \simeq \{p\} \times \PP^3.
$$
On the other hand, the conormal variety of the point in $\PP^5$ is 
$
\mathcal N_{\{p\}}
=
\{p\}\times\mathcal V(\xi_z)
\simeq \PP^4
$.
Hence $\mathcal N_{\{p\}}$ is not an irreducible component of the reduced special fiber $\bigl(\mathcal N_{\mathcal X/\Ab^1}\bigr)^{\rm red}_0$.
\hfill~$\diamond$  \end{example}
\subsection*{Multiplicities in the limiting conormal cycle}\label{subsec: mult}
Theorem~\ref{thm: limit conormal strata} describes the reduced specialization
of the relative conormal space~\eqref{eq: relative conormal}. Here we  recall
Sabbah's result~\cite{Sabbah1985}, presented in
Kleiman~\cite{Kleiman}, which governs the multiplicities of its irreducible
components.
Let $\mathcal S$ be a Whitney stratification of the flat map
$\pi:\mathcal X\to\Ab^1$, compatible with the special fiber $X_0$, and set
$
\mathcal S_0=\{S\in\mathcal S\mid S\subseteq X_0\}
$.
This is a finite set of strata. By Theorem~\ref{thm: limit conormal strata}, the
irreducible components of
$
(\mathcal N_{\mathcal X/\Ab^1})_0^{\rm red}
$
are among the conormal varieties
$
\mathcal N_{\overline S}
$
with $S\in\mathcal S_0$. Hence, as in~\eqref{eq: alg cycle}, we may write the corresponding algebraic cycle~as
\begin{equation}\label{eq: conormal as a cycle}
    \bigl[(\mathcal N_{\mathcal X/\Ab^1})_0\bigr]
=
\sum_{S\in\mathcal S_0}
m_S\,[\mathcal N_{\overline S}],
\end{equation}
where $m_S\in\mathbb Z_{\geq0}$ is the multiplicity of
$\mathcal N_{\overline S}$. If
$\mathcal N_{\overline S}$ is not an irreducible component, then~${m_S=0}$.

\smallskip

For $x\in X_0$, let $F_x$ denote the \emph{local Milnor fiber} (\cite[Sec.~4.2]{Dimca2004}) of $\pi$ at $x$ defined~as 
$$
F_x
\coloneqq
B_\delta^\circ(x)\cap \mathcal X \cap \pi^{-1}(t)
=
B_\delta^\circ(x)\cap X_t, \quad \text{for } \; 0<|t|\ll\delta\ll1.
$$
where $B_\delta^\circ(x)$ is the open
ball of radius $\delta$ centered at 
$(x,0)$ in a local coordinate chart of~$\PP^n\times\CC$. 

\smallskip

Recall that a constructible function on a complex algebraic variety $X$ is
a function $\alpha:X\to\mathbb Z$ which is constant on the strata of some
finite algebraic stratification of $X$. We write
$\operatorname{Eu}_{\mathcal X}$ for the local Euler obstruction
\cite{MacPherson1974} of the total space $\mathcal X$. It is
integer-valued, constant along Whitney strata for any Whitney stratification, and equal to $1$ on the
smooth locus; see, for instance, \cite{Dimca2004}. In particular,
$\operatorname{Eu}_{\mathcal X}$ restricts to a constructible function on
the local Milnor fiber $F_x$. We denote by
$\chi(F_x,\operatorname{Eu}_{\mathcal X})$ the Euler characteristic of
$F_x$ weighted by this function. Explicitly, if
$F_x=\bigsqcup_\alpha U_\alpha$ is some  Whitney stratification such that
$\operatorname{Eu}_{\mathcal X}=c_\alpha$ on $U_\alpha$, then it is given~by
$$
\chi(F_x,\operatorname{Eu}_{\mathcal X})
\coloneqq
\sum_\alpha c_\alpha\,\chi(U_\alpha).
$$
\begin{theorem}[Sabbah, see {\cite{Kleiman}}]\label{thm: sabbah mult}
Let $d=\dim X_0$, and choose one point $x_S\in S$ for every
$S\in\mathcal S_0$. The multiplicities $m_S$ in
\eqref{eq: conormal as a cycle} are determined by the triangular system
\begin{equation}\label{eq: sabbah}
\chi(F_{x_S},\operatorname{Eu}_{\mathcal X})
=
(-1)^d
\sum_{\substack{T\in\mathcal S_0, \;  S\subseteq \overline T}}
(-1)^{\dim T}\,
m_T\,\operatorname{Eu}_{\overline T}(x_S),
\qquad S\in\mathcal S_0 .
\end{equation}
Here, $\operatorname{Eu}_{\overline T}(x_S)$ denotes the local Euler
obstruction of $\overline T$ at $x_S$.
\end{theorem}
\begin{proof}
Sabbah's formula, presented in~\cite{Kleiman}, gives the 
pointwise identity for every $x\in X_0$. Since local Euler obstructions
are constant along Whitney strata, it is enough to evaluate it at the
chosen points $x_S$. The terms with $x_S\notin\overline T$ vanish,
leaving the triangular system~\eqref{eq: sabbah}.
\end{proof}
\begin{corollary}\label{cor: top components multiplicity}
Assume that $X_0$ is reduced. Then every top-dimensional
irreducible component of $X_0$ contributes its conormal variety  to the limit
$[({\cal N}_{\mathcal X/\mathbb A^1})_0]$ with multiplicity one.
\end{corollary}
\begin{proof}
Take $x\in S\subseteq X_0^{\rm sm}$. Since $\pi$ is smooth near $x$, the local Milnor fiber $F_x$ is a
ball with~$\chi(F_x,\operatorname{Eu}_{\mathcal X})=1$. The only stratum
$T$ with $S\subseteq\overline T$ is $T=S$, and
$\operatorname{Eu}_{\overline S}(x)=1$. Thus, 
$$
1=(-1)^d \cdot (-1)^d m_S=m_S
$$
by Sabbah's formula~\eqref{eq: sabbah}, which proves the statement.
\end{proof}
\begin{remark}
    Sabbah's formula~\eqref{eq: sabbah} does not imply Theorem~\ref{thm: limit conormal strata}.
Rather, it is most useful in combination with that theorem: Theorem~\ref{thm: limit conormal strata}
provides a finite list of candidate  components in the limiting conormal cycle, and Sabbah's
formula then determines their~multiplicities. In particular, for
a candidate ${\cal N}_{\overline{S}}$ that does not  occur,~\eqref{eq: sabbah}
gives $m_S=0$, as in Examples~\ref{ex: skew lines},~\ref{ex: nonpersistent point}.
\end{remark}
\begin{corollary}\label{cor: smooth general fibers}
Assume, in the setting of Theorem~\ref{thm: sabbah mult}, that general fibers $X_t$ for $t\neq 0$ are smooth. Let $\chi(F_x)$ be the Euler characteristic of the Milnor fiber $F_x$. Then~\eqref{eq: sabbah} becomes
\begin{equation}\label{eq: sabbah new}
\chi(F_x)
=
(-1)^d \sum_{T\in\mathcal S_0, \,  S\subseteq \overline T} (-1)^{\dim T} m_T\, \operatorname{Eu}_{\overline T}(x),
\quad x\in S.    
\end{equation} 
\end{corollary}
\begin{proof}
The local Milnor fiber
$
F_x=B_\delta^\circ(x)\cap \mathcal X \cap \pi^{-1}(t)
$
is contained in the nearby fiber~$X_t$ for some sufficiently small
$t \neq0$. By assumption, the fiber $X_t$ is smooth and hence every point of~$F_x$
lies in the smooth locus of $\mathcal X$, where
the local Euler obstruction ${\rm Eu}_{\mathcal X}$ is equal to $1$.
\end{proof}
The next result translates the multiplicities in the limiting conormal
cycle into the exponents of the irreducible factors  in the limit of the
dual hypersurface, as in~Example~\ref{ex: umbrella}.
\begin{proposition}\label{prop: multiplicities in dual}
Let $\pi:\mathcal X\to\Ab^1$ be a flat family with 
non-defective general fibers $X_t$ for $t\neq0$. 
Consider the corresponding degeneration of the dual
hypersurfaces $(X^\vee)_t$. Then, 
\begin{equation}\label{eq: dual as a cycle}
[(X^\vee)_0]
=
\sum_{\substack{S\in\mathcal S_0\\ \overline S\ \text{is non-defective}}}
m_S\,[\, (\overline{S})^\vee],    
\end{equation}
where the multiplicities $m_S$ are computed by Sabbah's formula~\eqref{eq: sabbah}.  
\end{proposition}
\begin{proof}
    We need to argue that the multiplicities are preserved after projecting 
$
(\mathcal N_{\mathcal X/\Ab^1})_0
$
to the dual projective space $(\PP^n)^\vee$. In the language of algebraic cycles, this is
the pushforward of the cycle~\eqref{eq: conormal as a cycle} under the second 
projection ${\rm pr}_2$. By definition, $({\rm pr}_2)_*[\mathcal{N_{\overline{S}}}] = 0$ when the map is not generically finite, that is, $\overline S$ is
dual defective. If $\overline S$ is non-defective, then it contributes
$$
m_S \cdot \deg({\rm pr}_2) \, [\, (\overline{S})^\vee].
$$
By \cite[Ch.\,1, Proposition 3.2]{Gelfand1994}, the projection ${\rm pr}_2$ is birational on every component~${\cal N}_{\overline S}$. Moreover, by biduality, if $(\overline{S}_1)^\vee = (\overline{S}_2)^\vee$, then $\overline{S}_1 = \overline{S}_2$, so two distinct non-defective irreducible strata cannot have
the same dual hypersurface. Thus, $\deg(\pr_2)=1$, and no multiplicities have to be combined. Pushing
forward the cycle~\eqref{eq: conormal as a cycle} gives the claimed formula.
\end{proof}
\begin{example}
We revisit Example~\ref{ex: umbrella} once more. Since the nearby fibers are
smooth, Corollary~\ref{cor: smooth general fibers} allows us to use ordinary
Euler characteristics of local Milnor fibers.
We use the stratification from Example~\ref{ex: whitneyUmb revisit}. The values needed in Sabbah's formula~\eqref{eq: sabbah new} are
\begin{equation}\label{tab: EuObs}
 \begin{array}{|c|c|c|c|c|c|}
\toprule  
\text{stratum $S$ containing }x
&
{\rm Eu}_{X_0}(x) & {\rm Eu}_{L}(x) & {\rm Eu}_{\{p_i\}}(x) & {\rm Eu}_{\{q_j\}}(x)
&
\chi(F_x)
\\
\midrule
X_0^{\rm sm} & 1 & 0 & 0 & 0 & 1\\
L\setminus \{p_i, q_j\} & 2 & 1 & 0 & 0 & 0\\
\{p_1, p_2\} & 1 & 1 & 1 & 0 & 2\\
\{q_1,q_2,q_3\} & 2 & 1 & 0 & 1 & 1 \\
\bottomrule
\end{array}   
\end{equation} 
The local Euler obstruction values  can be computed  using the
command \texttt{eulerObsMatrix} in \texttt{Macaulay2}. We now briefly explain
the Euler characteristic values of the local Milnor fibers.

\smallskip

At a smooth point $x$ of $X_0$, the local Milnor fiber is a contractible ball, hence
$\chi(F_x)=1$. At a general point of $L$, the functions $y_2$, $y_3$ 
and $y_2^3+y_3^3$ are units. Therefore the family 
$${\cal X}={\cal V}(y_1^2y_2-y_0^2y_3+t^2(y_3^3+y_2^3))\subset\PP^3\times \Ab^1 $$
is locally of the form
${\cal V}(ab+t^2)$ for $t \neq 0$, which is isomorphic to  $\CC^\times$  and has zero Euler characteristic. 
At any point $q_i$, we have $y_2,y_3\neq 0$, while
$y_2^3+y_3^3$ vanishes along the line~$L$. Let $s$ be a local
coordinate on $L$ centered at $q_i$, so the local equation of $F_{q_i}$ is given by $ab+t^2s = 0$. Since $t \neq 0$, the coordinate $s$ is uniquely determined by $a,b$ and hence the Milnor fiber is isomorphic to $\Ab^2_{a,b}$ with $\chi(\Ab^2) = 1$. Finally, consider $p_1$, the point $p_2$ is analogous. On the chart
$y_3\neq0$, we similarly set
$
a=y_0,\ b=y_1,\ \lambda=y_2/y_3
$, so 
the local Milnor fiber is given by
$$
a^2=\lambda b^2+t^2(1+\lambda^3).
$$
Projection to the plane $\Ab^2_{b,\lambda}$ is a finite-to-one map.  A general fiber over the base locus $B$ has two points, while over
the exceptional locus
$
E= \Ab^2_{b,\lambda} \setminus B = \mathcal V\bigl(\lambda b^2-t^2(1+\lambda^3)\bigr)
$
the fiber is one point. Applying multiplicativity and additivity of Euler characteristic, we~have
$$
\chi(F_{p_1}) = \chi(F_{p_2})
=
2\,\chi(B)+1\cdot\chi(E)
=
2\chi(\Ab^2)-\chi(E).
$$
It remains to compute $\chi(E)$. Since $1+\lambda^3$ is a unit near
$\lambda=0$, the curve $E$ is locally a circle
$
\lambda b^2=t^2
$ of Euler characteristic equal to
$\chi(E) =\chi(\CC^\times)=0$, and~thus
$$
\chi(F_{p_1})=\chi(F_{p_2})=2.
$$
Sabbah's triangular linear system has coefficients from Table~\eqref{tab: EuObs} and reads as
$$
\begin{aligned}
1 &= m_{X_0},\qquad \qquad \qquad \quad \, \, \text{ for } \; x \in X_0^{\rm sm},\\
0 &= 2m_{X_0}-m_L, \qquad \qquad \, \text{ for } \; x \in L\setminus \{p_i,q_j\},\\
2 &= m_{X_0}-m_L+m_{p_i},\quad  \;\;\;  \text{ for } \; x \in \{p_1,p_2\},\\
1 &= 2m_{X_0}-m_L+m_{q_i},\quad \; \, \text{ for } \; x \in \{q_1,q_2,q_3\}.
\end{aligned}
$$
Thus, we have
$
m_{X_0}=1,\,
m_L=2,\,
m_{p_i}=3,\,
m_{q_j}=1\,
$, recovering the exponents~in~\eqref{eq: limit Whitney}. \hfill~$\diamond$ 
\end{example}

\section{Dualizing Special Families}\label{sec: dializing special}
We now apply the general theory to several families where the limiting
conormal cycle can be written explicitly. We begin with the
classical case of hypersurfaces with isolated singularities. 

\smallskip

For a hypersurface in $\PP^n$ with an isolated singularity $p$, let
$\mu_p^{n}$ denote its \emph{Milnor number}, and let $\mu_p^{n-1}$ denote the
Milnor number of a general hyperplane section through $p$. For the definition and  properties of Milnor numbers, we refer to~\cite{Milnor, Dimca1992}. 
Note that for numerical experiments, there exist algorithms computing 
Milnor numbers, see for instance~\cite{MixedMultiplicityM2}.
\begin{proposition}[Isolated hypersurface singularities] \label{prop: IHS}
 Let
$
\mathcal X\to\Ab^1
$
be a flat family of~hypersurfaces such that the fibers $X_t\subset\PP^n$ are smooth for $t\neq0$, and the
special fiber $X_0$ has only isolated singularities
$p_1,\dots,p_r$. Then the specialization of the relative conormal has the~form  
\begin{equation}\label{eq: IHS conormal}
 \bigl[(\mathcal N_{\mathcal X/\mathbb A^1})_0\bigr]
=
[\mathcal N_{X_0}]
+
\sum_{j=1}^r \bigl( \mu_{p_j}^{n}+\mu_{p_j}^{n-1} \bigr) \, [\mathcal N_{p_j}].   
\end{equation}
\end{proposition}
\begin{proof}
By Corollary~\ref{cor: top components multiplicity}, the main conormal
component ${\cal N}_{X_0}$ appears with multiplicity one.
Let $m_p$ be the multiplicity of $[\mathcal N_p]$ in $\bigl[(\mathcal N_{\mathcal X/\mathbb A^1})_0\bigr]$. Sabbah's formula~\eqref{eq: sabbah new}
computed at $p$ gives
\begin{equation}\label{eq: milnor}
  \chi(F_p)
=
{\rm Eu}_{X_0}(p)+(-1)^{n-1} m_p \quad \Longleftrightarrow \quad m_p=(-1)^{n-1}\bigl(\chi(F_p)-{\rm Eu}_{X_0}(p)\bigr).  
\end{equation}
By Milnor's theorem~\cite{Milnor} (see also~\cite[Theorem~2.7]{Dimca1992}), the local Milnor fiber $F_p$ of an isolated hypersurface
singularity in $\CC^n$ is homotopy equivalent to a bouquet of $\mu_p^n$ many $(n-1)$-spheres.  In particular,
$
\chi(F_p)=1+(-1)^{n-1}\mu_p^n
$. On the other hand, one has
$
{\rm Eu}_{X_0}(p)=1+(-1)^{n}\mu_p^{n-1}
$, see~\cite[(6.2.2)]{LeTeissier1988}. Substituting this into~\eqref{eq: milnor} gives the claim.
\end{proof}
\begin{example}[Plane curves]
Let $C_0\subset\PP^2$ be a reduced plane
curve. For an isolated singularity~$p$, let $\delta_p$ denote the \emph{delta invariant}~\cite[p.~209]{Fischer2001} and let $r_p$ be the number of local
branches of~$p$. The  celebrated  formula of Milnor derived in~\cite{Milnor} then gives 
\begin{equation}\label{eq: Milnor formula plane}
   \mu_p^2=2\delta_p-r_p+1,
\quad
\mu_p^1=\operatorname{mult}_p(C_0)-1, 
\end{equation}
where the second equality follows from~\cite[p.~11]{Dimca1992}.
Hence coefficients in~\eqref{eq: IHS conormal}
become
$
2\delta_p+\operatorname{mult}_p(C_0)-r_p
$.
In particular, if the flat family $\pi: {\cal C} \to \Ab^1$ of smooth plane curves degenerates to $C_0$ with exactly $\tau$ nodes and $\kappa$ cusps, then the limiting conormal cycle~is
$$
\bigl[(\mathcal N_{\mathcal C/\Ab^1})_0\bigr]
=
[\mathcal N_{C_0}]
+
2\sum_{i=1}^{\tau}[\mathcal N_{p_i}]
+
3\sum_{j=1}^{\kappa}[\mathcal N_{q_j}].
$$
\end{example}
\begin{remark}
    As a corollary of Propositions~\ref{prop: multiplicities in dual} and~\ref{prop: IHS}, one can recover the degree of the dual hypersurface of a
hypersurface $X_0$ of degree $D$ with only isolated singularities $p_1,\dots,p_r$:
$$
\deg(X_0^\vee)
=
D(D-1)^{n-1}
-
\sum_{j=1}^r\bigl(\mu_{p_j}^n+\mu_{p_j}^{n-1}\bigr).
$$
For a different  proof of the above formula we refer to~\cite{KleimanEnumerativeSingularities}. In particular, for a plane curve,
$$
\deg(C_0^\vee)=D(D-1)-\sum_{j=1}^r (2\delta_j +{\rm mult}_{p_j}(C_0)- r_j),
$$
which recovers the \emph{General Class Formula} from \cite[Section A.5.4]{Fischer2001}.
\end{remark}
\begin{example}
Consider a flat family of smooth cubic surfaces degenerating to a \emph{Cayley~cubic}:
$$
\mathcal X
=
\mathcal V\bigl(
y_0y_1y_2+y_0y_1y_3+y_0y_2y_3+y_1y_2y_3
+
t(y_0^3+y_1^3+y_2^3+y_3^3)
\bigr)
\subset \PP^3\times\Ab^1_t.
$$
It has four isolated singular points. These are the following ordinary double points:
$$
p_0=[1:0:0:0],\quad
p_1=[0:1:0:0],\quad
p_2=[0:0:1:0],\quad
p_3=[0:0:0:1].
$$
At each node $p_i$ we compute sectional Milnor numbers
$
\mu_{p_i}^{3}=1$,
$\mu_{p_i}^{2}=1$
using the package \texttt{MixedMultiplicity}~\cite{MixedMultiplicityM2}. By
Propositions~\ref{prop: multiplicities in dual}, \ref{prop: IHS}, the limit of the dual
hypersurface~is
$$
 z_0^2z_1^2z_2^2z_3^2 \cdot \Delta_{X_0^\vee}.
$$
Here $X_0^\vee$ is a quartic surface, classically known as the \emph{Steiner
surface}, $\Delta_{X_0^\vee}$ is its defining polynomial, and
$z_0,z_1,z_2,z_3$ define the four planes dual to the nodes
$p_0,p_1,p_2,p_3$, respectively.
\end{example}

\subsection*{Generic Complete Intersections}\label{sec: CI}

We now pass to degenerations in which a smooth irreducible general fiber breaks into a transverse union of smooth irreducible components. We begin with a family studied in~\cite{Zhang}.
\begin{example}\label{ex: GCI}
Take a family of smooth hypersurfaces
$
{\XX={\cal V}(F) = {\cal V}(fg+th)\subset \PP^n\times\Ab^1_t}
$,
where $f,g,h$ are general homogeneous polynomials and
$\deg h=\deg f+\deg g$. 
Then, 
$$
X_0={\cal V}(fg)={\cal V}(f)\cup {\cal V}(g).
$$
Since $f,g,h$ are general, the divisors ${\cal V}(f)$, ${\cal V}(g)$, and ${\cal V}(h)$
meet transversely, so their union is a simple normal crossing divisor in $\PP^n$. Thus, $X_0$ has a natural stratification by intersections:
$$
S_f ={\cal V}(f)\setminus {\cal V}(g),\quad
S_g = {\cal V}(g)\setminus {\cal V}(f),\quad
S_{fg} = {\cal V}(f,g)\setminus {\cal V}(h),\quad
S_{fg}^h = {\cal V}(f,g,h).
$$
Together with the open stratum $\XX\setminus X_0$, this gives the
stratification of the map $\XX\to\Ab^1$. 

By Theorem~\ref{thm: limit conormal strata}, every irreducible component of
the special fiber of the relative conormal space is the conormal variety of one
of the strata above. We claim that, as an algebraic cycle,
\begin{equation}\label{eq: conormal GCI example}
    [({\cal N}_{\mathcal X/\mathbb A^1})_0]
=
[{\cal N}_{\overline{S}_f}]+[{\cal N}_{\overline{S}_g}]
+2\, [{\cal N}_{\overline{S}_{fg}}]
+[{\cal N}_{\overline{S}_{fg}^h}].
\end{equation}
By Corollary~\ref{cor: top components multiplicity}, smooth top-dimensional strata $S_f$ and $S_g$  contribute with multiplicity~one. 
We compute the remaining multiplicities as lengths of local algebras at
general points~as~in~\eqref{eq: alg cycle}. 
 We use the hypersurface case of Remark~\ref{rem: relative conormal ideal}.
For a family~${\cal V}(F)$, with fiber coordinates
$(y_0,\ldots,y_n)$ and dual coordinates $(z_0,\ldots,z_n)$ the
relative conormal ideal is 
given~by
\begin{equation}\label{eq: hypersurface relative conormal}
    \left(
(F) + 
I_2
\begin{pmatrix}
z_0 & z_1 & \cdots & z_n\\
\partial_{y_0}F & \partial_{y_1}F & \cdots & \partial_{y_n}F
\end{pmatrix}
\right)
:
(\partial_{y_0}F,    \ldots, \partial_{y_n}F)^\infty
:
(z_0,\ldots,z_n)^\infty.
\end{equation}
In the local computations below, we work on the affine chart
$y_0\neq 0$ of $\mathbb P^n$, with coordinates
$u_i=y_i/y_0$ for $i=1,\ldots,n$. A hyperplane
$[\eta_0:\cdots:\eta_n]\in(\mathbb P^n)^\vee$ is then determined by
$$
\eta_0=-(\eta_1 u_1 + \ldots + \eta_n u_n).
$$
Thus, locally, we consider the dual projective
coordinates
$
[\eta_1:\cdots:\eta_n]\in\mathbb (\PP^{n-1})^\vee
$.
The function $h$ is a unit on the stratum~$S_{fg}$. Therefore, we can suppose that in a
neighbourhood of a generic point of $S_{fg}$, the family $\mathcal X$ is
given by $u_1u_2+t=0$. Then, the stratum $S_{fg}$ is given by ${\cal V}(u_1,u_2)$ and the fiber gradient is
$\nabla_uF=(u_2,u_1,0,\ldots,0)$. We pass to the affine chart
$\eta_1=1$ on $\mathbb (\PP^{n-1})^\vee$ and write $\eta_2=\alpha$. By
\eqref{eq: hypersurface relative conormal}, the local relative~conormal~ideal~is 
\[
J_{fg}:\Sigma_{fg}^{\infty}, \quad \text{where } \; J_{fg}=
\left(
u_1u_2+t,\,
u_1-\alpha u_2,\,
\eta_j u_1,\,
\eta_j u_2 \text{ for } j\geq 3
\right),
\quad \text{and } \; 
\Sigma_{fg}=(u_1,u_2,t).
\]
Here $\Sigma_{fg}$ is the singular ideal of the fiber. Thus, near a general point of ${\cal N}_{\overline{S}_{fg}}$, the
conormal space is given by $\left(
u_1u_2+t,\,
u_1-\alpha u_2,\,
\eta_3,\ldots,\eta_n
\right)$. Taking the special fiber at $t=0$, we get
$$
R_{fg}\coloneqq
\frac{
\CC[u_1,\ldots,u_n,\alpha]
}{
(u_1u_2,\, u_1-\alpha u_2)
}.
$$
Let $\zeta$ be the generic point of ${\cal N}_{\overline{S}_{fg}}$ corresponding to the
prime $(u_1,u_2)$ and residue field~$\mathcal K$.~Then,
$$
(R_{fg})_\zeta
\simeq
 \frac{\mathcal K[u_1,u_2]}{(u_1u_2,\, u_1-\alpha u_2)}
\simeq
\mathcal K[u_2]/(u_2^2), \quad \mathcal{K}=\CC(u_3,\ldots,u_n,\alpha).
$$
Thus $\operatorname{length}(R_{fg})_\zeta=2$, so ${\cal N}_{\overline{S}_{fg}}$ appears
with multiplicity $2$. Along $S^h_{fg}$, the local equation is
$
u_1u_2+tu_3=0
$.
The stratum is locally ${\cal V}(u_1,u_2,u_3)$ and the fiber gradient $\nabla_uF$ is
$
(u_2,u_1,t,0,\ldots,0)
$.
We pass to the 
chart $\eta_1=1$ on $(\mathbb P^{n-1})^\vee$ and write~${(\eta_1:\eta_2:\eta_3)=(1:\alpha:\beta)}$.
Saturating
by the singular ideal gives $\eta_4=\cdots=\eta_n=0$. Let
$\mathcal K=\CC(u_4,\ldots,u_n,\alpha,\beta)$. Then,
$$
\frac{\mathcal K[u_1,u_2,u_3,t]}{J_h:\Sigma_h^\infty}
\simeq
\frac{\mathcal K[u_1,u_2,u_3,t]}
{(u_1-\alpha\beta^{-1}t,\, u_2-\beta^{-1}t,\, u_3+\alpha\beta^{-2}t)},
$$
where
$
J_h=(u_1u_2+tu_3,\, u_1-\alpha u_2,\, t-\beta u_2,\, \alpha t-\beta u_1)$ and $\Sigma_h=(u_1,u_2,t)$. Taking $t=0$ gives
$$
\mathcal K[u_1,u_2,u_3]/(u_1,u_2,u_3)\simeq\mathcal K.$$ 
Hence the component
${\cal N}_{\overline{S}^h_{fg}}$ appears with multiplicity one. This completes the analysis.
\hfill~$\diamond$  \end{example}

\begin{remark}\label{rem: length}
In the example above, the algebra
$
A=\mathbb K[x]/(x^2)
$
is Artinian, with maximal ideal $\mathfrak m=(x)$ and residue
field
$
A/\mathfrak m \simeq \mathbb K
$.
Hence its length equals its dimension as a $\mathbb K$-vector space,
see, e.g.,~\cite[Sec.~7, Exercise~1.6]{Liu2002}. The same applies to
$
{A=
\mathbb K[x_1,\ldots,x_k]/(x_1^{m_1},\ldots,x_k^{m_k})}
$.
\end{remark}
Example~\ref{ex: GCI} above admits a direct generalization. Consider a family
$
\mathcal X=\mathcal V(F_1,\ldots,F_k)$ in $\PP^n\times\Ab^1_t$,
with 
$F_i=f_{i1}\cdots f_{ir_i}+t h_i
$
and
$
\deg h_i=\sum_{j=1}^{r_i}\deg f_{ij}.
$
Assume that the divisors $\mathcal V(f_{ij})$ and $\mathcal V(h_i)$ are
smooth and have transverse intersections, and that
$\pi:\mathcal X\to\Ab^1$ is a flat family of complete intersections with smooth general fibers. For each $i = 1,\dots,k$, let $J_i\subseteq\{1,\ldots,r_i\}$ and
take $\varepsilon_i\in\{0,1\}$. We define the corresponding
locally closed stratum as
\begin{equation}\label{eq: stratum GCI}
S_{\mathbf{J},\boldsymbol{\varepsilon}}
=
\left(
\bigcap_{i=1}^k\ \bigcap_{j\in J_i} \mathcal V(f_{ij})
\;\cap\;
\bigcap_{\varepsilon_i=1} \mathcal V(h_i)
\right)
\setminus
\left(
\bigcup_{i=1}^k\ \bigcup_{j\notin J_i} \mathcal V(f_{ij})
\;\cup\;
\bigcup_{\varepsilon_i=0} \mathcal V(h_i)
\right).    
\end{equation}
\begin{lemma}\label{lem: CI mult}
In the above notation, let
$
\mathcal X=\mathcal V(F_1,\ldots,F_k)\subset \PP^n\times\Ab^1_t
$.
For a stratum~$S_{\mathbf J,\boldsymbol\varepsilon}$,~set
$$
m_i (\mathbf J,\boldsymbol\varepsilon)=
\begin{cases}
|J_i|, & \varepsilon_i=0,\\
|J_i|-1, & \varepsilon_i=1.
\end{cases}
$$
Then, as an algebraic  cycle, the special fiber of the relative conormal
space is given by
$$
\bigl[(\mathcal N_{\mathcal X/\mathbb A^1})_0\bigr]
=
\sum_{\mathbf J,\boldsymbol\varepsilon}
\bigl(\prod_{i=1}^k m_i(\mathbf J,\boldsymbol\varepsilon)\bigr)
\bigl[{\cal N}_{\overline{S}_{\mathbf J,\boldsymbol\varepsilon}}\bigr],
$$
where empty strata and terms with
coefficient $0$ are omitted.   
\end{lemma}
\begin{proof}
    Similar to Example~\ref{ex: GCI}, the transversality condition implies that near every
stratum~$S_{\mathbf{J},\boldsymbol{\varepsilon}}$ the family is locally a complete intersection of independent
blocks in disjoint sets of variables. Therefore letting $u_{ij} = x_{ij}/x_{i1}$ be coordinates on the affine chart $x_{i1} \neq 0$ of $\PP^n$, we can suppose that  near $S_{\mathbf{J},\boldsymbol{\varepsilon}}$, the family is given by the following equations:
\begin{align*}
    u_{i1}\cdots u_{is_i}+t=0, \quad \varepsilon_i=0, \quad \text{with } s_i = |J_i|,\\ 
    u_{i1}\cdots u_{is_i}+t u_{i,s_i+1}=0, \quad \varepsilon_i=1, \quad \text{with } s_i = |J_i|.
\end{align*} 
One can check Thom's $a_\pi$-condition~(Definition~\ref{def: thom})  directly from the tangent
space equations of these local models. 
It remains to compute the~multiplicities of the strata~\eqref{eq: stratum GCI}.

\smallskip

Since the variables of different blocks are disjoint, after choosing a
general conormal direction, the local algebra~$R$, which computes the
multiplicity, splits block by block. Explicitly, let $R_i$ be the local coordinate ring of the relative conormal space for the block
${\cal V}(F_i)$, then 
$$
R = R_1\otimes_{\mathbb K[t]}\cdots\otimes_{\mathbb K[t]}R_k,
$$
where $
\mathbb K=\mathbb C(\alpha_{ij})
$
is the rational function field generated by general conormal directions in all blocks. The rest is analogous to Example~\ref{ex: GCI}: the relative conormal ideal of ${\cal V}(F_i)$ is locally
$$
\left( (F_i) + 
I_2
\begin{pmatrix}
1 & \alpha_{i2} & \cdots & \alpha_{is_i}\\
\partial_{u_{i1}}F_i & \partial_{u_{i2}}F_i & \cdots &
\partial_{u_{is_i}}F_i
\end{pmatrix} \right) :
(\partial_{u_{i1}}F_i,    \ldots, \partial_{u_{is_i}}F_i)^\infty
$$
which, after setting $t=0$, implies $u^{s_i} = 0$ for $\varepsilon_i=0$ and 
$u^{s_i-1} = 0$ for
$\varepsilon_i=1$, where $ u \coloneq u_{i1}$. Thus,
\[
R/(t) \simeq \mathbb{K}[u_1]/(u_1^{|J_1| - \varepsilon_1}) \otimes_{ \mathbb{K}} \cdots \otimes_{ \mathbb{K}} \mathbb{K}[u_k]/(u_k^{|J_k| - \varepsilon_k}).
\]
By Remark~\ref{rem: length}, its  length is 
$\dim_{\mathbb{K}} R/(t) = \prod_{i=1}^k m_i (\mathbf J,\boldsymbol\varepsilon)$, which completes the~proof.
\end{proof}

\subsection*{Reciprocal Linear Spaces}\label{sec: RLS}
In this subsection we study degenerations of the  dual varieties of reciprocal linear
spaces. This gives a recursive  proof of the degree formula conjectured in~\cite{matsubaraheo2026principalmatroiddeterminants}. A different proof of this result 
has previously appeared in \cite{JulianLeonie, Speyer2009}. Here, we avoid Whitney stratifications and
Sabbah's formula, and instead exploit the underlying deletion-contraction structure.

\medskip

We recall the definition. Let $L\subset \PP^n$ be a linear space of dimension $d<n$. The
\emph{reciprocal linear space} $\mathcal R_L$ is the Zariski closure of the image of
$L$ under the Cremona map
$$
\operatorname{crem}:\PP^n\dashrightarrow \PP^n,\quad
(z_0:\cdots:z_n)\mapsto (z_0^{-1}:\cdots:z_n^{-1}).
$$
If $A$ is a full-rank $(d+1)\times(n+1)$ matrix whose row span is $L$,
we denote by $M(L)$ the linear matroid on $\{0,\ldots,n\}$ represented by the
columns of $A$. Recall that the vanishing ideal~$I_{\mathcal R_L}$ is generated by the circuit polynomials of the matroid 
$M(L)$, and these form a universal Gr\"obner basis, see
\cite{ProudfootSpeyer2006} for details.
One of the goals of this section is the following formula.
\begin{theorem}\label{thm: RecipDeg}
Let $\mathcal R_L^\vee$ be a hypersurface. Then its degree depends only on $M(L)$ and~equals
\begin{equation}\label{eq: RL degree}
\deg(\mathcal R_L^\vee)=2^d  \cdot \beta(M(L)),    
\end{equation}
where $\beta(M(L))$ is the beta invariant of the matroid $M(L)$.
\end{theorem}
\noindent
We first record an observation on Gröbner degenerations of hypersurfaces with special weights.
\begin{proposition}\label{lem: specialWeightLimit}
Let $X={\cal V}(F)$ be an irreducible hypersurface in $\PP^n$ not contained in
${H_n={\cal V}(z_n)}$. Consider $\omega=(1,\ldots,1,0)$, and let $X_0$ be the special
fiber of the Gröbner dege\-neration of $X$ with respect to $\omega$. Then,
scheme-theoretically, the special fiber $X_0$ equals 
\begin{equation}\label{eq: cone}
   X_0=\operatorname{Cone}_p(X\cap H_n),
\quad
p=[0:\cdots:0:1]. 
\end{equation}
\end{proposition}
\begin{proof}
Since $F$ is homogeneous, the monomials of maximal
$\omega$-weight are precisely the monomials not involving $z_n$. Hence
$
\operatorname{in}_\omega(F)=F(z_0,\ldots,z_{n-1},0)
$,
which defines the cone~\eqref{eq: cone}.    
\end{proof}
We write $L\setminus\{n\}$
for deletion, which is the image of $L$ under the coordinate projection forgetting $z_n$,  and $L/\{n\}$ for contraction, which is the hyperplane section $L\cap\{z_n=0\}$.
\begin{theorem}\label{thm: limitRLS}
Let $(\mathcal R_L)_t$ be the flat Gr\"obner degeneration of
$\mathcal R_L$ with respect to the weight $w=(0,\ldots,0,1)$, and assume that
$\mathcal R_L^\vee$ is a hypersurface. Then, set-theoretically, the special fiber of the induced
Gr\"obner degeneration  of the dual hypersurface with respect to~$-\omega$ equals
\begin{equation}\label{eq: RL dual set-theoretic}
    (\mathcal R_L^\vee)_0
=
\operatorname{Cone}_p(\mathcal R_{L\setminus\{n\}}^\vee)
\;\cup\;
\operatorname{Cone}_p(\mathcal R_{L/\{n\}}^\vee),
\end{equation}
where $p=[0:\cdots:0:1]\in(\PP^n)^\vee$. Moreover, as an algebraic cycle, we have
$$
[(\mathcal R_L^\vee)_0]
=
[\operatorname{Cone}_p(\mathcal R_{L\setminus\{n\}}^\vee)]
+
2\, [\operatorname{Cone}_p(\mathcal R_{L/\{n\}}^\vee)],
$$
where terms corresponding to defective dual varieties are omitted.
\end{theorem}
\begin{proof}
By Corollary~\ref{cor: dual deg}, the dual degeneration of $\mathcal R_L^\vee$ is the Gröbner
degeneration  with respect to
$-w = (0,\ldots,0,-1)$, equivalently, one can take
$v = (1,\ldots,1,0)$. By Lemma~\ref{lem: specialWeightLimit},
$$
(\mathcal R_L^\vee)_0
=
\operatorname{Cone}_p(\mathcal R_L^\vee\cap H_n).
$$
It remains to identify $\mathcal R_L^\vee\cap H_n$. By
\cite[Proposition~2.8]{JulianLeonie}, the rational map
$$
\Phi:L\times L^\perp\dashrightarrow
\PP^n\times(\PP^n)^\vee,
\qquad
(x,y)\longmapsto
\bigl(\operatorname{crem}(x),\,x\star x\star y\bigr),
$$
where $\star$ denotes coordinate-wise Hadamard product, parametrizes the
conormal variety~$\mathcal N_{\mathcal R_L}$. The first component
$\operatorname{crem}(x)$ recovers $x$, so $\Phi$ is birational onto its
image. Moreover, by \cite[Ch.\,1, Proposition 3.2]{Gelfand1994}, for non-defective reciprocal linear spaces, the second
projection
$
\pr_2:\mathcal N_{\mathcal R_L}\to\mathcal R_L^\vee
$
is generically one-to-one.
Thus, the following composition
$$
\phi\coloneq \pr_2\circ\Phi:
L\times L^\perp\dashrightarrow(\PP^n)^\vee,
\quad
(x,y)\mapsto x\star x\star y,
$$
is a birational parametrization of the dual $\mathcal R_L^\vee$, which also appears in~\cite[Proposition~4.1]{matsubaraheo2026principalmatroiddeterminants}.
The last coordinate of $\phi(x,y)$ is
$
z_n=x_n^2y_n
$.
Therefore the preimage of $H_n={\cal V}(z_n)$ is the union of the two divisors~$
{\cal V}(x_n)\cup {\cal V}(y_n)
$. If $y_n=0$, then $y\in L^\perp\cap H_n=(L\setminus\{n\})^\perp$, and the
restricted parametrization~$\phi$ gives
$
\mathcal R_{L\setminus\{n\}}^\vee\times\{0\}.
$
If $x_n=0$, then $x\in L\cap H_n$, which identifies with $L/\{n\}$, and
the restricted parametrization~$\phi$ gives
$
\mathcal R_{L/\{n\}}^\vee\times\{0\}.
$
Thus, we obtain~\eqref{eq: RL dual set-theoretic}. 

\smallskip

We now determine the multiplicities. Let
$
F\in \KK[z_0,\ldots,z_n]
$
be the defining equation~of~$\mathcal R_L^\vee$. From the set-theoretic
description above, the initial form  
$
{\rm in}_vF = F|_{z_n=0}$ factorizes as $F_D^{m_1}F_C^{m_2}
$,
where $F_D$ and $F_C$ define the deletion and contraction components,
respectively.
We write
$$
F(z_0,\ldots,z_n)
=
F_0(z')+z_nF_1(z')+z_n^2F_2(z')+\cdots,
\quad
z'=(z_0,\ldots,z_{n-1}),
$$
with
$
F_0=F|_{z_n=0}=F_D^{m_1}F_C^{m_2}.
$
Since $F$ vanishes on the image of~$\phi$
and
$
\phi_n(x,y)=x_n^2y_n
$,
we~get
\begin{equation}\label{eq: pullback}
 F_0(\phi'(x,y))
=
-x_n^2y_n\,F_1(\phi'(x,y))
+
O\bigl((x_n^2y_n)^2\bigr), \quad \phi'(x,y)=(x_0^2y_0,\ldots,x_{n-1}^2y_{n-1}).   
\end{equation}
We now explain why the multiplicities $m_1,m_2$ can be read from the expansion~\eqref{eq: pullback}. The
multiplicity $m_1$, respectively $m_2$, is the order of vanishing of $F_0$ at
a general point of the deletion or contraction component, respectively. Since
$\phi$ is birational onto $\mathcal R_L^\vee$, and its restrictions to
$\mathcal V(y_n)$ and $\mathcal V(x_n)$ are birational onto $\mathcal R_{L\setminus\{n\}}^\vee$ and
$\mathcal R_{L/\{n\}}^\vee$ respectively, this order of vanishing can be computed after pulling back to the
parameter space. Equivalently,
$$
m_1=\operatorname{ord}_{\mathcal V(y_n)}(F_0\circ\phi'),
\quad
m_2=\operatorname{ord}_{\mathcal V(x_n)}(F_0\circ\phi').
$$

\smallskip

It remains to check that $F_1$ is a unit at general points of both
components. Indeed, differentiating
$F(\phi(x,y))=0$ in tangent directions to $L$ and $L^\perp$ gives the following
$$
x\star y\star \nabla F(\phi(x,y))\in L^\perp,
\qquad
x^2\star \nabla F(\phi(x,y))\in L.
$$
At a general point these conditions determine the gradient uniquely up to
scale, namely
\begin{equation}\label{eq: differential}
    \nabla F(\phi(x,y))\propto (1/x_0:\cdots:1/x_n).
\end{equation}
On the deletion component $y_n=0$, we have $x_n\neq0$ generically, so
$F_1 = \partial_{z_n}F\neq0$.
On the contraction component
$x_n=0$, with $y_n\neq0$, taking the projective limit in~\eqref{eq: differential} as
$x_n\to0$ gives
$
\nabla F\propto e_n,
$
so again $F_1 = \partial_{z_n}F\neq0$, and $F_1$ is a unit on both components.

\smallskip

Now~\eqref{eq: pullback} gives the orders immediately. Along ${\cal V}(y_n)$, the function $x_n$
is a unit and $y_n$ is a local equation. Since $F_1$ is a unit,
$F_0\circ\phi'$ has a simple zero. Thus
$
m_1=1
$.
Along ${\cal V}(x_n)$, the function $y_n$ is a unit and $x_n$ is a local
equation. Since the leading term in~\eqref{eq: pullback} is $x_n^2$ times a unit,
$F_0\circ\phi'$ vanishes to order $2$. Thus
$
m_2=2
$, which completes the proof.
\end{proof}
\begin{proof}[Proof of Theorem~\ref{thm: RecipDeg}]
Assume that $\mathcal R_L$ is not dual defective. By \cite[Theorem~2.11]{JulianLeonie}, this happens if and only if $M(L)$ is connected.
We prove the formula by induction on $n$. 

The case $n=1$ is
immediate: $L$ is a point in $\PP^1$, so $\mathcal R_L^\vee$ is a hyperplane
in $(\PP^1)^\vee$, and
$$
\deg(\mathcal R_L^\vee)=1=2^0\beta(U_{1,2}).
$$
We now pass to the induction step. Since $M(L)$ is connected,
$n$ is neither a loop nor a coloop. By Theorem~\ref{thm: limitRLS} and
flatness of the dual Gröbner degeneration, we have the equality:
\begin{equation}\label{eq: recDeg recursive}
\deg(\mathcal R_L^\vee)
=
\deg(\mathcal R_{L\setminus\{n\}}^\vee)
+
2\deg(\mathcal R_{L/\{n\}}^\vee).    
\end{equation}
The deletion $L\setminus\{n\}$ has dimension $\dim L = d$, while the contraction
$L/\{n\}$ has dimension $\dim L - 1 = d-1$. Applying the induction hypothesis to~\eqref{eq: recDeg recursive} gives
$$
\deg(\mathcal R_L^\vee)
=
2^d\beta(M(L\setminus\{n\}))
+
2\cdot 2^{d-1}\beta(M(L/\{n\})) = 2^d\beta(M(L)),
$$
where the last equality follows from the deletion-contraction identity for the beta invariant by
\cite[Theorem~1]{CRAPO1967406}.  This is precisely the claimed degree formula~\eqref{eq: RL degree}.
\end{proof}
\begin{example}
We now illustrate Theorems~\ref{thm: RecipDeg} and~\ref{thm: limitRLS} on a reciprocal linear space. Take
$$
A=
\begin{small}
\begin{pmatrix}
1&0&0&1&1\\
0&1&0&1&2\\
0&0&1&1&3
\end{pmatrix}\end{small},
\quad
L=\operatorname{rowspan}(A)\subset\PP^4.
$$
We consider the Gr\"obner degeneration of $\mathcal R_L$ with respect to the weight~$
\omega=(0,0,0,0,1).
$
The circuit polynomials form a
universal Gr\"obner basis for $I_{\mathcal R_L}$ by~\cite[Theorem~4]{ProudfootSpeyer2006}. Thus,
\begin{align*}
   I_{\mathcal{R}_L} = \bigl( &
   y_1y_2y_3-y_1y_2y_4-2y_1y_3y_4-y_2y_3y_4,\;
   y_0y_1y_2-3y_0y_1y_4-2y_0y_2y_4-y_1y_2y_4,
       \\
      &
   y_0y_2y_3-2y_0y_2y_4-y_0y_3y_4+y_2y_3y_4,\;
   y_0y_1y_3-3y_0y_1y_4+y_0y_3y_4+2y_1y_3y_4
   \bigr).
\end{align*}
A direct computation gives the minimal prime decomposition
$
 {\rm in}_\omega I_{\mathcal R_L}
=
(y_4,I_D)\cap I_C,
$
where
\begin{align*}
  I_D & = I_{\mathcal R_{L\setminus\{4\}}} =
\bigl(
y_0y_1y_3+y_0y_2y_3+y_1y_2y_3-y_0y_1y_2
\bigr), \\
I_C & =
I_{\mathcal R_{L/\{4\}}} =
\bigl(
y_1y_2+2y_1y_3+y_2y_3,\,
3y_0y_1-y_0y_3-2y_1y_3,\,
y_2y_3-2y_0y_2-y_0y_3
\bigr).
\end{align*}
which are the deletion and contraction components, respectively.
Thus, the special fiber decomposes as
$
(\mathcal R_L)_0
=
\bigl(H_4\cap\mathcal R_{L\setminus\{4\}}\bigr)
\cup
\mathcal R_{L/\{4\}}
$.
The dual hypersurface $\mathcal R_L^\vee$ has defining polynomial $\Delta_{\mathcal R_L}$ of degree $12$ containing $1818$ terms. This agrees with the statement of~Theorem~\ref{thm: RecipDeg}:
$$
\deg(\mathcal R_L^\vee)= 2^{\dim L}\beta(M(L)) = 2^2\cdot \beta(U_{3,5})=2^2 \cdot {5-2 \choose 3-1} =2^2\cdot3=12.
$$
The initial form of $\Delta_{\mathcal R_L}$ with respect to the opposite weight $-\omega$ factors as
$$
\operatorname{in}_{-\omega}(\Delta_{\mathcal R_L}) = \Delta_{\mathcal R_{L\setminus\{4\}}}\,(\Delta_{\mathcal R_{L/\{4\}}})^2,$$ where
 $\Delta_{\mathcal R_{L\setminus\{4\}}}$ and $\Delta_{\mathcal R_{L/\{4\}}}$ are both of degree $4$.
This agrees with Theorem~\ref{thm: limitRLS}. \hfill~$\diamond$  \end{example}

\subsection*{Toric degenerations}\label{sec: Grass}

The purpose of this subsection is to illustrate that toric
degenerations are not automatically simple from the point of view of
limiting conormal cycles. For a toric variety, the decomposition into
torus orbits is a Whitney stratification, as discussed in~\cite[\S 11.5.B]{Gelfand1994} and \cite{ElHilanyHelmerTsigaridas2025}.
However, this stratification need not be minimal. Moreover, 
the Whitney
stratification of the total family compatible with the special fiber may
refine the torus-orbit stratification. Additionally, not every torus orbit should be expected to give a component of the
limiting conormal cycle. Therefore, even when the special fiber is toric, the description of the limiting conormal cycle is not straightforward. We illustrate this with a toric degeneration of $\Gr(1,\PP^3)$.


\begin{example}
\label{ex: GT toric}
The Grassmannian $\Gr(1,\PP^3)\subset \PP^5$ is the quadric
$
{p_{14}p_{23}-p_{13}p_{24}+p_{12}p_{34}=0}
$.
Consider its flat degeneration ${\cal X}$ with special fiber
$
X(1,3)=\mathcal V(p_{14}p_{23}-p_{13}p_{24}),
$ which is 
a toric hypersurface with singular locus
$
L=\mathcal V(p_{13},p_{24},p_{14},p_{23})
$. It is the only singular torus orbit of~$X(1,3)$.
However, the Whitney stratification of the total family 
refines $L$ by the~points
$$
p_1=L\cap \mathcal V(p_{12}), \quad
p_2=L\cap \mathcal V(p_{34}). $$
These points arise from singularities of the total family at $t=0$.
Thus, by Theorem~\ref{thm: limit conormal strata},
$$
[(\mathcal N_{\mathcal{X}/\mathbb A^1})_0]
=
[\mathcal N_{X(1,3)}]
+
m_L[\mathcal N_L]
+
m_{p_1}[\mathcal N_{p_1}]
+
m_{p_2}[\mathcal N_{p_2}].
$$
A direct computation in \texttt{Macaulay2} gives
$
m_L=2, \, m_{p_1}=m_{p_2}=1
$, see~\cite{zenodo}.
\end{example}

The example above is the first case of the \emph{Gelfand-Tsetlin toric
degeneration}~\cite[Sec.~14]{miller2005combinatorial} of the Grassmannian $\Gr(k,\PP^n)$. This is a Gröbner
degeneration with respect to a certain weight~$\omega$. In general, the
Gelfand-Tsetlin special fibers $X(k,n)$ are \emph{Grassmann-Hibi toric
varieties}: their toric ideals are the \emph{Hibi ideals}~\cite{Hibi1987} of the distributive
lattices of Plücker variables, and their polytopes are \emph{marked order
polytopes}~\cite{Ardila2011}. The singular orbit closures of $X(k,n)$ are themselves
smaller Hibi varieties, and are described combinatorially~in~\cite{Brown2010Singular}.

\smallskip

Example~\ref{ex: GT toric} suggests that the minimal Whitney stratification of the Gelfand-Tsetlin family is obtained by
refining the singular torus-orbit strata using the lower-order terms of the
Plücker relations. Namely, for a Plücker relation $P$, define its
$\omega$-tail by
$$
\operatorname{tail}_\omega(P)
=
\operatorname{in}_\omega\bigl(P-\operatorname{in}_\omega(P)\bigr).
$$
We expect that  the extra components of the corresponding limiting conormal
cycle arise from singular torus-orbit strata and from their nonempty intersections
with such tail hypersurfaces.

\begin{problem}\label{prob: GT}
Describe the minimal Whitney stratification of the Gelfand-Tsetlin family
compatible with its toric special fiber, and compute the corresponding
conormal multiplicities.
\end{problem}



\section{Relative Higher Associated Hypersurfaces}\label{sec: HighAssociated}
In this section we study higher associated hypersurfaces under flat Gröbner
degenerations. 
Let
$
\pi:\mathcal X\to\Ab^1
$
be a flat Gröbner family with reduced irreducible general fibers
$
X_t\subset\PP^n
$. Assume also that the higher associated variety $\mathcal Z_i(X_t)$~\eqref{eq: assocHyper} is a
hypersurface for every~$t\neq0$. We first define the \emph{higher associated hypersurface  relative  to  $\pi$},
in analogy with 
Definition~\ref{def: rel conormal}.
\begin{definition}
  The \emph{higher associated
hypersurface relative  to $\pi$} is defined as
$$
\mathcal Z_i(\mathcal X/\Ab^1)
\coloneqq
\overline{
\bigcup_{t\in \Ab^1\setminus\{0\}}\mathcal Z_i(X_t)
}
\subset
\Gr(n-d+i-1,\PP^n)\times\Ab^1.
$$
\end{definition}
\noindent
The goal is to describe the special fiber
$
\bigl(\mathcal{Z}_i(\mathcal X/\mathbb A^1)\bigr)_0 
$.
This provides a natural generalization of
Theorem~\ref{thm: limit conormal strata} to higher associated
hypersurfaces. In particular, it tells us that, in order to determine the
decomposition of the special fiber
$\bigl({\cal Z}_i(\mathcal X/\mathbb A^1)\bigr)_0$, it is enough to know the
minimal Whitney stratification of the family ${\cal X}$, 
together
with the multiplicities arising from Sabbah's formula~\eqref{eq: sabbah}.
One of the key roles in this transition is played by the 
\emph{Cayley trick}:
\begin{equation}\label{eq: cayleytrick}
 \overline{\rho^{-1} ({\cal Z}_i(X))}
=
Y^\vee, \quad \text{with} \;\; Y = \sigma_i(X\times \mathbb P^{d-i}),   
\end{equation}
where
$
\sigma_i:\mathbb P^n\times \mathbb P^{d-i}
\hookrightarrow
\mathbb P^{(n+1)(d-i+1)-1}
$
is the Segre embedding and  the Plücker  map
$$
\rho:
\mathbb P\!\left(\operatorname{Mat}_{(d-i+1)\times(n+1)}\right)
\dashrightarrow
\operatorname{Gr}(n-d+i-1,\mathbb P^n)
$$
sends Stiefel coordinates to Plücker coordinates. See~\cite[Chapter~3, Theorem 2.7]{Gelfand1994}. 

\smallskip

We will use the following notation for the weight induced on Plücker
coordinates. Let
$
p_I$ be the Plücker coordinates on
$\Gr(n-d+i-1,\PP^n)$. Given the weight
$\omega=(\omega_0,\ldots,\omega_n)\in\mathbb Z^{n+1}$, we denote by
$
\omega^{(i)}\in \mathbb Z^{\binom{n+1}{d-i+1}}
$
the induced Plücker weight defined by
$
\omega^{(i)}(p_I)
\coloneqq
\sum_{k\in I}\omega_k$. Let again $\mathcal S$ be the minimal Whitney stratification of the total space
$\mathcal X$ compatible with $X_0$, and
$$
\mathcal S_0=\{S\in\mathcal S\mid S\subseteq X_0\}.
$$
\begin{theorem}\label{thm:higher-associated-limit}
Consider a flat Gr\"obner degeneration
$
\pi:\mathcal X\to\Ab^1
$
with respect to the weight $\omega\in\mathbb Z^{n+1}$, with fibers
$X_t\subset\PP^n$ of dimension $d$, and fix
$i\in\{0,\ldots,d\}$. Assume that the higher associated variety
$\mathcal Z_i(X_t)$ is a hypersurface for all $t\neq 0$. Then the
induced family
$
\mathcal Z_i(\mathcal X/\Ab^1)
$
is a flat Gr\"obner degeneration with respect to the induced 
weight $-\omega^{(i)}$ on $\CC[p_I]$.
Moreover, 
\begin{equation}\label{eq: higher associated limit}
 \bigl[
\bigl(\mathcal Z_i(\mathcal X/\Ab^1)\bigr)_0
\bigr]
=
\sum_{\substack{
S\in\mathcal S_0,\ j_S\geq 0\\
j_S\leq \dim S-\codim (\overline S)^\vee+1
}}
m_S
\bigl[
\mathcal Z_{j_S}(\overline S \, )
\bigr],
\quad
j_S=i-d+\dim S,   
\end{equation}
where the multiplicities $m_S$ are computed via Sabbah's
formula~\eqref{eq: sabbah} on $\mathcal X$.
\end{theorem}
Before proving Theorem~\ref{thm:higher-associated-limit}, let us record
several immediate consequences.
\begin{corollary}\label{cor:higher-associated-no-extra}
Assume that $X_0$ is reduced and regular in codimension $i$. Then no singular
stratum contributes to the decomposition in
Theorem~\ref{thm:higher-associated-limit}. Hence
$
\bigl[
\bigl(\mathcal Z_i(\mathcal X/\mathbb A^1)\bigr)_0
\bigr]
=
[\mathcal Z_i(X_0)].
$
\end{corollary}
\noindent
 For $i=0$ and $i=1$, this recovers limits of Chow and Hurwitz forms, studied in~\cite{Dalbec1995, sturmfels2016hurwitzformprojectivevariety}, respectively. In the two cases below  the limit is
independent of the chosen degeneration. Thus, Chow and Hurwitz forms may
be defined intrinsically for reducible algebraic cycles. 
\begin{corollary}[{\cite{Dalbec1995}}]\label{cor:chow-hurwitz-cycles}
Let
$
[X_0]=\sum_\alpha n_\alpha [Y_\alpha]
$
be a pure $d$-dimensional algebraic cycle in $\PP^n$. For any flat Gröbner
degeneration $\mathcal X\to\mathbb A^1$ whose special fiber has
fundamental cycle $[X_0]$, the special fiber of the relative Chow form
satisfies
$$
\bigl(\operatorname{Ch}(\mathcal X/\mathbb A^1)\bigr)_0
=
\prod_\alpha \operatorname{Ch}(Y_\alpha)^{n_\alpha}.
$$
\end{corollary}
\begin{corollary}\label{cor: chhurdeg}
Let
$
[X_0]=\sum_\alpha [Y_\alpha]
$
be a reduced pure $d$-dimensional algebraic cycle in $\PP^n$, with each
$Y_\alpha$ regular in codimension one. Let $Z$ run over the
codimension-one components of the singular locus $\operatorname{Sing}(X_0)$. Suppose
$\mathcal X\to\mathbb A^1$ is a flat Gröbner degeneration with smooth
general fiber and special fiber $X_0$. Then the
special fiber of the relative Hurwitz form satisfies
$$
\bigl(\operatorname{Hu}(\mathcal X/\mathbb A^1)\bigr)_0
=
\prod_\alpha \operatorname{Hu}({Y_\alpha})
\prod_Z \operatorname{Ch}({Z})^{2\delta_p(C_0)}, 
$$
where $p\in Z$ is a general point and $L$ is a general linear space
of codimension $d-1$ through $p$, such that
$
C_0:=X_0\cap L
$
is a sectional curve with the delta invariant $\delta_p(C_0)$ of $p$.
\end{corollary}

\begin{proof}
 The set-theoretic description follows from
Theorem~\ref{thm:higher-associated-limit}. It remains to compute the exponents $m_Z$ of $\operatorname{Ch}(Z)$.   
Take general $p\in Z$, and let $r_p$ be the number of components
$Y_\alpha$ passing through $p$. Since every $Y_\alpha$ is regular in
codimension one, all $Y_\alpha$ 
are smooth at $p$. Hence all local
Euler obstructions appearing in Sabbah's formula~\eqref{eq: sabbah new} are equal to one. Thus, 
\begin{equation}\label{eq: hurwitz cycle}
  \chi(F_p)
=
(-1)^d\bigl( \, 
\sum_{\alpha:\, p\in Y_\alpha}(-1)^d
+
(-1)^{d-1}m_Z
\bigr)
=
r_p-m_Z.  
\end{equation}
Now choose a general linear space $L$ of codimension $d-1$ through
$p$, which is transverse to~$Z$. By Remark~\ref{rem: whitney properties}, the induced stratification
$\mathcal S\cap L$ is again Whitney. Since $p$ is a general point of
the stratum $Z$, the linear space $L$ is a normal slice to $Z$ at $p$.
Thus, $F_p$ is homeomorphic to the product of a small
ball $B\subset Z$ with the restricted Milnor fiber $F_p\cap L$, see,
for instance, discussion in~\cite{Massey2022}. The latter is the Milnor fiber of the family of sectional curves
$
\pi|_L: \mathcal X\cap (L\times \mathbb A^1)\to \mathbb A^1
$
whose special fiber is
$
C_0=X_0\cap L
$. In particular, we get
$$
\chi(F_p)=\chi(B) \cdot \chi(F_p\cap L)=\chi(F_p\cap L).
$$
The curve $C_0$ has $r_p$ branches. Milnor's formula~\eqref{eq: Milnor formula plane} (see also~\cite[Theorem 2.2]{Bassein1977}) gives
$$
\mu_p(C_0)=2\delta_p(C_0)-r_p+1,
$$
where $\mu_p(C_0)$ is the Milnor number of the reduced curve singularity, see~\cite{BuchweitzGreuel1980} for definition. The
Milnor fiber of $C_0$ is homotopy equivalent to a bouquet of
$\mu_p(C_0)$ circles, and therefore
$$
\chi(F_p)
=
1-\mu_p(C_0)
=
r_p-2\delta_p(C_0).
$$
Comparing this with~\eqref{eq: hurwitz cycle}, we obtain
$
m_Z=2\delta_p(C_0)
$.
This gives the claimed factorization.
\end{proof}

\begin{example}
    We revisit the Gelfand-Tsetlin degeneration of the Grassmannian $\Gr(k,\PP^n)$ from~\eqref{ex: GT toric}.
    The  polytope $P_{k,n}$ of the toric variety $X(k,\PP^n)$ is called the \emph{Gelfand-Tsetlin~polytope}. By~\cite[Theorem~1.4]{Ardila2011}, it is a
marked order polytope, and hence it is normal by~\cite{stanley1986twoposet}.
Thus $X(k,n)$ is normal, in particular regular in
codimension one.
By Corollary~\ref{cor:higher-associated-no-extra}, we
obtain
$$
\bigl(\operatorname{Hu}(\Gr(k,\PP^n))\bigr)_0
=
\operatorname{Hu}(X(k,n)).
$$
Thus, in this case, no extra Chow factors appear in the limit of the Hurwitz form. In particular, the Hurwitz degree of $\Gr(k,\PP^n)$ equals the
Hurwitz degree of the toric variety $X(k,n)$. Let
$$
r=\dim \Gr(k,\PP^n)=(k+1)(n-k).
$$
Since the toric variety $X(k,n)$ is normal, by Bertini's
theorem a general linear section of codimension~$(r-1)$ avoids its singular locus and is a smooth curve.  By~\cite[Theorem~1.1]{sturmfels2016hurwitzformprojectivevariety},
$$
\operatorname{Hdeg}(\Gr(k,\PP^n))
=
\operatorname{Hdeg}(X(k,n))
=
2\deg X(k,n)+2g(X(k,n))-2,
$$
where $g(X(k,n))$ is the sectional genus of $X(k,n)$. Kushnirenko's theorem~\cite{Kushnirenko1976} gives:
$$
\deg X(k,n)=r!\operatorname{Vol}(P_{k,n}).
$$
Moreover, by Khovanskii's formula~\cite[Theorem~1]{Khovanskii1978} for complete intersections,
we obtain
$$
g(X(k,n))
=
1+\sum_{q=1}^{r-1}
(-1)^q
\binom{r-1}{q}
\operatorname{Ehr}(P_{k,n}, -q).
$$
Thus, $\operatorname{Hdeg}(\Gr(k,\PP^n))$ can be read off
combinatorially from the Gelfand-Tsetlin polytope.
   \hfill~$\diamond$ 
\end{example}

The rest of this section is devoted to the proof of
Theorem~\ref{thm:higher-associated-limit}. To translate the problem to
the setting of Section~\ref{sec: dualizing flat}, we use the Cayley
trick~\eqref{eq: cayleytrick} and record two elementary facts: a Gröbner degeneration of a variety $X\subseteq \PP^n$ induces a Gröbner degeneration of its Segre
embedding $\sigma_i(X\times \PP^k) \subseteq \PP^{(n+1)(k+1)-1}$, and the corresponding weights in Stiefel and Plücker coordinates  are related. This is the content of the following two propositions.
\begin{proposition}\label{prop: GbBasisSegre}
Let $\mathcal X\to \mathbb A^1$ be the Gröbner degeneration of
$X\subset \PP^n$ with respect to $\omega$. Set
$$
{Y_t}=\sigma_i(X_t\times \PP^{d-i})
\subset \PP^{(n+1)(d-i+1)-1} \quad \text{with } \; Y\coloneqq Y_1.
$$
Then the family
$
\mathcal Y:=\sigma_i(\mathcal X\times \PP^{d-i})\to \mathbb A^1
$
is the Gröbner degeneration of $Y$ with respect to the weight~$\widetilde{\omega}$ with
$
\widetilde\omega_{kj}=\omega_k
$ for $0\leq k\leq n$ and $0\leq j\leq d-i$
on the Segre coordinates $z_{kj}$.
\end{proposition}
\begin{proof}
Take  $x=[x_0:\cdots:x_n]\in X$ and 
$y=[y_0:\cdots:y_{d-i}]\in \PP^{d-i}$. The Segre embedding is given by
$
z_{kj}=x_k y_j.
$
Over $t\neq 0$, the Gröbner degeneration of $X$ is induced by the
torus action $$x_k\mapsto t^{\omega_k}x_k,$$ while the $y$-variables  are fixed.
Hence
$
z_{kj}=x_k y_j
\mapsto
t^{\omega_k}x_k y_j
=
t^{\widetilde\omega_{kj}}z_{kj}.
$
Thus each fiber $Y_t$ is obtained from the variety $Y$ by the torus action with the weight
$\widetilde\omega$. Taking the closure  gives the full
Gröbner degeneration of $Y$ with respect to $\widetilde\omega$. Equivalently, on the level of ideals, by~\cite[Theorem~2.9]{sullivant2006}, if $G$ is a 
Gröbner basis of $I_X$, then the lifts of
the elements of~$G$, together with the $2\times 2$ minors of the
Segre matrix $(z_{kj})$, form a Gröbner basis of $I_Y$.
\end{proof}
We now pass from dual Stiefel coordinates to Plücker coordinates. We write
$$
a_{jk}, \quad \text{for} \quad 0\leq j\leq d-i,\quad 0\leq k\leq n,
$$
for the 
coordinates on the dual projective space 
$
\left(\PP^{(n+1)(d-i+1)-1}\right)^\vee
=
\PP\!\left(\operatorname{Mat}_{(d-i+1)\times(n+1)}\right).
$
Thus a matrix $A=(a_{jk})$ defines the linear space
$
L=\PP(\ker A)$ in the Grassmannian $\Gr(n-d+i-1,\PP^n)
$.
The Plücker embedding $\rho: \CC[p_{I}] \to \CC[a_{jk}]$ is given by
$$
p_I\mapsto \det(a_{jk})_{k\in I},
\quad \text{for } \; |I|=d-i+1.
$$
By an analogous argument as in Proposition~\ref{prop: GbBasisSegre}, the following result is straightforward.
\begin{proposition}\label{lem: weightPltoStiefel}
Let $-\widetilde\omega$ be the weight on dual Stiefel coordinates
$a_{jk}$ defined by
$
-\widetilde\omega_{jk}=-\omega_k$,
for 
$0\leq j\leq d-i$ and $ 0\leq k\leq n
$.
A Gröbner degeneration in dual Stiefel coordinates with weight
$-\widetilde\omega$ induces the
Gröbner degeneration in Plücker coordinates with weight
$-\omega^{(i)}$, where
$$
\omega^{(i)}_I=\sum_{k\in I}\omega_k.
$$
\end{proposition}

\begin{proof}[Proof of Theorem~\ref{thm:higher-associated-limit}]
Let
$
\mathcal Y:=\sigma_i(\mathcal X\times \PP^{d-i})
\subset \PP^{(n+1)(d-i+1)-1}\times \Ab^1 .
$
By Proposition~\ref{prop: GbBasisSegre}, this is the Gröbner degeneration of
$Y=\sigma_i(X\times\PP^{d-i})$ with weight
$\widetilde\omega_{jk}=\omega_k$. Hence, by~Corollary~\ref{cor: dual deg}, the dual family
$\mathcal Y^\vee$ degenerates with 
$-\widetilde\omega$ in dual Stiefel coordinates. By the Cayley trick~\eqref{eq: cayleytrick},
$$
\overline{\rho^{-1}(\mathcal Z_i(X_t))}=Y_t^\vee .
$$
By Proposition~\ref{lem: weightPltoStiefel}, the family
$\mathcal Z_i(\mathcal X/\Ab^1)$ is then  the Gröbner degeneration with
 weight~$-\omega^{(i)}$.

\smallskip

We now describe the special fiber. By the product property of Whitney
stratifications from Remark~\ref{rem: whitney properties}, the
strata $S\times\PP^{d-i}$ with $S\in\mathcal S$, form a Whitney
stratification of $\mathcal X\times\PP^{d-i}$. Applying the Segre
embedding gives a minimal Whitney stratification of $\mathcal Y$, compatible~with~$Y_0$. We now apply Proposition~\ref{prop: multiplicities in dual}
to the family $\mathcal Y$ (where we omit defective strata closures):
$$
[(\mathcal Y^\vee)_0]
=
\sum_{S\in\mathcal S_0}
m_S
\left[
\bigl(\sigma_i(\overline S\times\PP^{d-i})\bigr)^\vee
\right].
$$
Let $s=\dim S$ and set
$
j_S=i-d+s.
$
Since $d-i=s-j_S$, the Cayley trick applied to $\overline S$ gives
$$
\bigl(\sigma_i(\overline S\times\PP^{d-i})\bigr)^\vee
=
\overline{\rho^{-1}\bigl(\mathcal Z_{j_S}(\overline S)\bigr)}.
$$
By~\cite[Ch.~1, Cor.~5.9]{Gelfand1994}, this dual is hypersurface
if and only if
$
{0\leq j_S\leq
s-\codim(\overline S)^\vee+1}
$.
Passing from Stiefel to Plücker coordinates then yields exactly the decomposition in~\eqref{eq: higher associated limit}.

\smallskip

It remains to identify the multiplicities. For
$(x,y)\in S\times\PP^{d-i}$, the Milnor fiber of 
$\mathcal Y\to\Ab^1$ is $F_x\times B$, with $B$ a small ball in
$\PP^{d-i}$, and local Euler obstruction is unchanged under product with a
smooth factor:
$
\operatorname{Eu}_{\overline T\times\PP^{d-i}}(x,y)
=
\operatorname{Eu}_{\overline T}(x).
$
Moreover, in Sabbah's formula~\eqref{eq: sabbah}, dimensions of the fiber and of all the strata
increase by $d-i$, so the signs are unchanged. Hence Sabbah's
system for $\mathcal Y$ is identical to that for
$\mathcal X$, and the coefficients are the same~$m_S$.
\end{proof}

\section{Mixed Discriminants for Horizontally Parametrized systems}\label{sec: discriminants}

We now apply the preceding results to discriminants of parametrized
polynomial systems. We use the framework of \emph{horizontally parametrized
systems} introduced in~\cite{HPsystems2024}. Fix collections of
polynomials 
$
\mathcal B_i=\{b_{i0},\ldots,b_{ik_i}\}\subset \KK[x_1,\ldots,x_n]$ for $
i=1,\ldots,n
$.
A horizontally parametrized system with parameters $a_{ij}$ and polynomial supports $\mathcal B_i$ is a system 
\begin{equation}\label{eq: horizontalSys intro}
f_i(x)=\sum_{j=0}^{k_i}a_{ij}b_{ij}(x)=0,
\quad i=1,\ldots,n.
\end{equation}
The system~\eqref{eq: horizontalSys intro} is   \emph{unmixed} if all supports
$\mathcal B_i$ coincide, \emph{fully mixed} if they are pairwise
distinct, and \emph{semi-mixed} if the supports fall into several blocks
of equal supports.
This generalizes sparse polynomial systems, where the supports consist
of monomials. The corresponding discriminant records parameter
values for which the system is non-generic, for instance has  a
multiple root in $(\KK^*)^n$. In the sparse case this recovers the
mixed discriminants of~\cite{Cattani2013}.
\begin{example}
Generically, the following system has four isolated solutions in $(\KK^*)^2$:
$$
a_0+a_1x+a_2y+a_3(x^2+y^2)
=
b_0+b_1x+b_2y+b_3(y^2+y)
=0
$$
 The discriminantal locus where two of these solutions collide is a hypersurface in
$\PP^3\times\PP^3$ of bidegree $(5,6)$. Its
defining polynomial has $334$ terms, starting with
$$
\Delta_{\rm mixed}
=
16a_3^5b_0^2b_1^4
-16a_1a_3^4b_0b_1^5
-4a_2^2a_3^3b_1^6
+\cdots
+16a_0a_3^4b_3^6.
\vspace{-0.75cm}
$$
   \hfill~$\diamond$ 
\end{example}
We now explain how the theory of the previous sections applies to such
systems. We first treat the unmixed case, so
$
\mathcal B_1=\cdots=\mathcal B_n=\{b_0,\ldots,b_k\}.
$
Choose a weight vector~$\omega$ on the $x$-variables and let $b_j^t$ be
the corresponding weighted homogenization of $b_j$. This gives maps
$$
\phi_t:\CC^n\dashrightarrow \PP^k,\quad
x\mapsto [b_0^t(x):\cdots:b_k^t(x)],
$$
with image closures $X_t=\overline{\operatorname{im}(\phi_t)}$. Set
$
Y_t=\sigma(X_t\times \PP^{n-1}).
$
Then, by the
Cayley trick~\eqref{eq: cayleytrick}, the dual~$Y_t^\vee$ is the Hurwitz hypersurface of~$X_t$, equivalently the
discriminant of the system with polynomial support~$\{b_j^t\}$. When $\omega$
induces a total order on $\{b_j\}$, then $b_j^0=\operatorname{in}_\omega(b_j)$
are monomials, hence the variety $X_0$ is toric and the limit system is sparse.
By~\cite[Theorem~11.4]{sturmfels1996grobner}, the family $X_t$ is a
Gröbner toric degeneration of $X$ precisely when
$
\{s b_0,\ldots,s b_k\}\subset \CC[s,x]
$
is a \emph{Khovanskii basis}~\cite{KavehManon2019} w.r.t.~the weight
$
\widehat\omega:=(1,\omega)
$.
In that case the induced weight on the coordinates of $\PP^k$ is
$
v=A^\top\widehat\omega
$,
where the columns of $A$ are the exponent vectors of the monomials
$
s\cdot\operatorname{in}_\omega(b_j)
$.
By Proposition~\ref{prop: GbBasisSegre}, the Segre varieties
$Y_t$ form a Gröbner degeneration with weight
$
u=(v,\ldots,v)
$.
Thus the results of Section~\ref{sec: HighAssociated} apply to the
corresponding discriminants.

\begin{remark}
An extension of this result to semi-mixed supports, and more generally to
positive-dimensional solution sets, would require a theory of
\emph{multigraded higher associated hypersurfaces}. For
the multigraded Chow and Hurwitz forms, see
\cite{OssermanTrager2019,DoganErgurTsigaridas2023,PrattSodomacoSturmfels2026}.
\end{remark}

\begin{example}[Graph of the Grassmannian]
The graph of the Grassmannian was studied in
\cite{borovik2024coupledclusterdegreegrassmannian} due to its
relation to \emph{coupled cluster equations} in quantum chemistry. This leads to
an eigenvalue problem on the Grassmannian, which is
equivalent to taking a general linear section of the graph.
Fix $n>k$, and consider the chart of $\Gr(k,\PP^n)$ of row spans of matrices
$$
X=
\begin{footnotesize}
 \begin{pmatrix}
1 & 0 & \cdots & 0 & x_{1,k+2} & \cdots & x_{1,n+1}\\
0 & 1 & \cdots & 0 & x_{2,k+2} & \cdots & x_{2,n+1}\\
\vdots & & \ddots & \vdots & \vdots & & \vdots\\
0 & 0 & \cdots & 1 & x_{k+1,k+2} & \cdots & x_{k+1,n+1}
\end{pmatrix}   
\end{footnotesize}
.
$$
After homogenizing with an extra variable $x_0$, the maximal minors of~$X$
define a rational map
$$
\mathbf i:\PP^{(k+1)(n-k)}\dashrightarrow
\PP^{\binom{n+1}{k+1}-1}.
$$
Let $\mathcal G(k,\PP^n)$ be the graph of this map. The discriminant of
the coupled cluster equations is the Hurwitz form
$\operatorname{Hu}(\mathcal G(k,\PP^n))$. Fix $k=1$.
By~\cite[Proposition~10]{borovik2024coupledclusterdegreegrassmannian},
there is a weight~$\omega$ for which the homogenized 
parametrization of $\mathcal G(1,\PP^n)$ forms a Khovanskii basis. This
gives a Gröbner toric degeneration of $\mathcal G(1,\PP^n)$. Let
$\mathcal T(1,n)$ be the toric special fiber.
Its polytope is the Cayley sum of the Gelfand-Tsetlin polytope with the
simplex corresponding to the coordinates on $\PP^{2(n-1)}$ and  admits a regular unimodular triangulation. Hence both the polytope and the toric variety
$\mathcal T(1,n)$ are normal by
\cite[Proposition~13.15]{sturmfels1996grobner}. Thus 
$\mathcal T(1,n)$ is regular in codimension one, and
Corollary~\ref{cor:higher-associated-no-extra} applies.  That is, if $v$ is the
induced weight on the graph coordinates and $u=(v,\ldots,v)$ is the
induced Stiefel weight, then we have
$$
\operatorname{in}_{-u}\operatorname{Hu}(\mathcal G(1,\PP^n))
=
\operatorname{Hu}(\mathcal T(1,n)).
$$
Thus, the coupled cluster discriminant degenerates to the sparse
$A$-discriminant of the~eigenvalue problem on $\mathcal T(1,n)$. Assuming \cite[Conjecture~14]{borovik2024coupledclusterdegreegrassmannian}, the
same holds  for any~$k$.~\hfill$\diamond$
\end{example}

\subsection*{Generalized Cayley configurations}

We now consider systems which are not necessarily unmixed. Similarly to  previous sections, to such systems we associate a projective variety,  
whose dual gives us the discriminantal locus. 

\smallskip

Following~\cite{Cattani2013}, a solution $u\in(\KK^*)^n$ of a horizontally parametrized system
\begin{equation*}
f_i(x)=\sum_{j=0}^{k_i}a_{ij}b_{ij}(x)=0,
\quad i=1,\ldots,n.
\end{equation*} is  a \emph{non-degenerate multiple
root} if the gradients
$
\nabla_x f_1(u),\ldots,\nabla_x f_n(u)
$
are linearly dependent, while any $n-1$ of them are linearly
independent. Equivalently, $\operatorname{Jac}_F(u)$ has corank~one.
The \emph{mixed discriminantal variety} is the closure of the set
of coefficients for which the system has a non-degenerate multiple root
in $(\KK^*)^n$. When this is a hypersurface, its defining equation
is the \emph{mixed discriminant}  $\Delta_{\rm mixed}$, otherwise the
system is called defective and we set $\Delta_{\rm mixed}=1$.

\smallskip

Introduce auxiliary variables $y=(y_1,\ldots,y_n)$ and set
$
G(x,y)=y_1f_1(x)+\cdots+y_nf_n(x).
$
We associate to $G$ the projective variety $X_G\subset \PP^m$,
where $m+1=\sum_i(k_i+1)$, parametrized~by
\begin{align*}
    \Phi:(\KK^*)^n\times \PP^{n-1} &\dashrightarrow \PP^m, \quad
    (x,y) \mapsto
[
y_1b_{10}(x):\cdots:y_1b_{1k_1}(x):
\cdots:
y_nb_{n0}(x):\cdots:y_nb_{nk_n}(x)
].
\end{align*}
Let $\Delta_G$ be the $X_G$-discriminant, i.e., the defining
equation of $X_G^\vee$ when this dual is a hypersurface, and
$\Delta_G=1$ otherwise. Equivalently, $\Delta_G$ vanishes whenever $\mathcal{V}(G)$ is singular.

\begin{theorem}[Generalized Cayley trick]\label{thm:genCayleyTrick}
Assume that, for each $i$, the base locus system 
$
b_{i0}(x)=\cdots=b_{ik_i}(x)=0
$
has no solution in $(\KK^*)^n$. Then
$
\Delta_G=\Delta_{\rm mixed}.
$
\end{theorem}

\begin{proof}
The argument follows closely the proof of~\cite[Theorem~2.1]{Cattani2013}.
Let $c=(c_{ij})$ be a general point of the mixed discriminantal
variety ${\cal V}(\Delta_{\rm mixed})$. Thus the following system
$$
f_i^c(x)=\sum_{j=0}^{k_i}c_{ij}b_{ij}(x),\quad i=1,\ldots,n,
$$
has a non-degenerate multiple root $u\in(\KK^*)^n$. Hence
$\operatorname{Jac}_F(u)$ has corank one, so we can~choose
$$
0\neq \lambda=(\lambda_1,\ldots,\lambda_n)
\in \ker(\operatorname{Jac}_F(u)^\top).
$$ 
Let
$
G^c(x,y)=\sum_{i=1}^n y_i f_i^c(x)
$.
Then
$
G^c(u,\lambda)=0$, since $\frac{\partial G^c}{\partial y_i}(u,\lambda)=f_i^c(u)=0$, and also
\begin{align*}
  \nabla_xG^c(u,\lambda)
=
\sum_{i=1}^n \lambda_i\nabla_xf_i^c(u)
=
\operatorname{Jac}_F(u)^\top\lambda
=
0.
\end{align*}
Thus $(u,[\lambda])$ is a singular point of the hypersurface
${\cal V}(G^c)$. Equivalently, the hyperplane $H_c$ is
tangent to $X_G$. Thus, the mixed discriminant  is
contained in $X_G^\vee$. In particular, if $\Delta_{\rm mixed}\neq 1$,
then $\Delta_{\rm mixed}$ divides $\Delta_G$. Since $X_G^\vee$ is
irreducible when it is a hypersurface, equality follows once both
discriminants are non-defective.
If both $\Delta_{\rm mixed}$ and $\Delta_G$ are defective, we still get
$$
\Delta_{\rm mixed} = \Delta_G.
$$
It remains to exclude the possibility that $X_G^\vee$ is a hypersurface
while the mixed discriminant  is defective. Assume that
$X_G^\vee$ is a hypersurface, and let $c\in X_G^\vee$ be general.
Choose a smooth point
$
q\in X_G
$
at which the corresponding hyperplane $H_c$ is tangent. Since $c$ is general,
we may assume that $q=\Phi(u,[\lambda])$ lies in the image of $\Phi$, and by the
hypothesis on the supports, for every $i$ the vector
$
(b_{i0}(u),\ldots,b_{ik_i}(u))
$
is nonzero. Tangency means precisely that $(u,[\lambda])$ is a critical
point of $G^c$. Hence
$u$ is a multiple root of the system, that is,
$$
f_1^c(u)=\cdots=f_n^c(u)=0 \quad \text{and} \quad 
\operatorname{Jac}_F(u)^\top\lambda=0.
$$
Suppose this multiple root is degenerate. Then some $n-1$ of the
vectors $\nabla_x f_i^c(u)$ are linearly dependent, say $
\nabla_x f_1^c(u),\ldots,\nabla_x f_{n-1}^c(u)
$. Equivalently, there
exists a nonzero vector
$$
\mu=(\mu_1,\ldots,\mu_{n-1},0)\in \ker(\operatorname{Jac}_F(u)^\top) \quad \text{such that } \quad \sum_{i=1}^{n-1}\mu_i \nabla_x f_i^c(u) = 0.
$$
 For all $s\in\KK$, outside
a finite set, the vector $\lambda+s\mu$ is nonzero and
$
\operatorname{Jac}_F(u)^\top(\lambda+s\mu)=0.
$
Therefore $(u,[\lambda+s\mu])$ is again a critical point of $G^c$.
The corresponding points
$
\Phi(u,[\lambda+s\mu])
$
are distinct for general $s$, since no support block
$(b_{i0}(u),\ldots,b_{ik_i}(u))$ vanishes. Hence the
hyperplane  $H_c$ is tangent to $X_G$ along a positive-dimensional set. Thus, the 
projection
$
\mathcal N_{X_G}\to X_G^\vee
$
is positive-dimensional, which contradicts the assumption that $X_G^\vee$
is non-defective. Thus, for general $c\in X_G^\vee$, the corresponding system has a
non-degenerate multiple root.
Hence ${\cal V}(\Delta_{\rm mixed})$    contains a dense open subset of $X_G^\vee$, and thus it is a
hypersurface. Therefore
the two discriminants coincide again. This completes the proof.
\end{proof}

\begin{example}[Duffing oscillator]
Harmonic balancing for the \emph{Duffing oscillator} gives~polynomial systems
whose algebraic geometry was studied in~\cite{BREIDING2025110492, BorovikBreiding2024}. The first~example~is
$$
a_1x(x^2+y^2)+a_2x+a_3y+a_4=0,\quad
b_1y(x^2+y^2)+b_2x+b_3y+b_4=0.
$$
Its discriminant $\Delta$ has degree $18$ and $578$ terms
by~\cite[Theorem~3.1]{BREIDING2025110492}. It is the defining equation of
the dual variety $X_G^\vee$ from the generalized Cayley construction parametrized by
$$
\{sx(x^2+y^2),\,sx,\,sy,\,s,\,ty(x^2+y^2),\,tx,\,ty,\,t\}.
$$
By~\cite[Lemma 5.4]{BorovikBreiding2024}, this set is a Khovanskii basis for the weight
$
\omega(s,t,x,y)=(1,1,2,1)
$.
The induced sparse system in the limit is
$
a_1x^3+a_2x+a_3y+a_4=0=
b_1xy^2+b_2x+b_3y+b_4=0.
$
Its $A$-discriminant $\Delta_A$ has degree $14$ and $90$ terms
and  appears as a factor of $\operatorname{in}_u(\Delta)$:
$$
\operatorname{in}_u(\Delta)=\Delta_A\,b_1^2a_2^2,
\quad
u=-A^\top\omega = -(7,3,2,1,6,3,2,1),
$$
where $u$ is the induced weight on the dual coordinates. The extra factors $b_1^2a_2^2$ arise from the dual hypersurfaces to
zero-dimensional strata in the singular locus~of the special fiber~$(X_G)_0$.~\hfill$\diamond$
\end{example}


\section*{Acknowledgements}
We are grateful to Luca Sodomaco, Bernd Sturmfels, Julian Weigert, and Leonie Kayser for helpful discussions. We especially thank Simon Telen for suggesting that the degeneration techniques of
this paper could be applied to reciprocal linear spaces.
This work was carried out during the master's internship of C.B. at Max Planck Institute for Mathematics in the Sciences in Leipzig. C.B. thanks MPI-MiS for its hospitality and the Fondation de l’ENS for financial support. 
The research of V.B. is funded by the European Union (ERC, UNIVERSE PLUS, 101118787). {\small Views and opinions expressed are, however, those of the authors only and do not necessarily reflect those of the European Union or the European Research Council Executive Agency. Neither the European Union nor the granting authority can be held responsible for them.}

\printbibliography
\bigskip 

\medskip 

\bigskip
\noindent
\small {\bf Authors' addresses:}
\smallskip

\noindent Viktoriia Borovik, MPI-MiS Leipzig, Germany
\hfill {\tt viktoriia.borovik@mis.mpg.de}

\noindent Clara Briand, École Normale Supérieure, France
\hfill {\tt clara.briand@ens.psl.eu}
\end{document}